\documentclass{emsprocart}
\usepackage{stmaryrd}
\usepackage{xr}
\usepackage{amscd, amssymb, calc, ifthen, mathrsfs}
\usepackage[matrix,arrow,curve,cmtip]{xy}
\usepackage{amsxtra}
\usepackage{eucal}
\usepackage{color}
\usepackage[normalem]{ulem}
\usepackage{amsxtra}
\usepackage{makeidx}

\contact[james.borger@anu.edu.au]{James Borger, Mathematical Sciences Institute, Australian National University, Canberra ACT 0200, Australia}

\let \: \colon
\newcommand{\mathbold}{\mathbb}
\newcommand{\bF}{{\mathbold F}}
\newcommand{\bA}{{\mathbold A}}
\newcommand{\bQ}{{\mathbold Q}}
\newcommand{\bQpl}{{\bQ_+}}
\newcommand{\bZ}{{\mathbold Z}}
\newcommand{\bC}{{\mathbold C}}
\newcommand{\bR}{{\mathbold R}}
\newcommand{\bRpl}{{\bR_+}}
\newcommand{\bRplex}[1]{\bR_{+}^{#1}}
\newcommand{\bN}{{\mathbold N}}
\newcommand{\ord}{{\mathrm{ord}}}
\newcommand{\GL}{{\mathbf{GL}}}

\newcommand{\vbl}{-}
\newcommand{\tn}{\otimes}           % Tensor
\DeclareMathOperator{\tr}{tr}
\DeclareMathOperator{\diag}{diag}
\newcommand{\dfn}{:=}
\newcommand{\setof}[2]{\{#1 \;|\; #2\}}
\newcommand{\bigsetof}[2]{\big\{#1 \;|\; #2\big\}}
\newcommand{\longmap}{{\,\longrightarrow\,}}

\def\longisomap{{\,\buildrel \sim\over\longrightarrow\,}} % an isomorphism
\newcommand{\eps}{\varepsilon}
\DeclareMathOperator{\colim}{\mathrm{colim}}

\newcommand{\Gal}{{\mathrm{Gal}}}
\newcommand{\sO}{{\mathcal{O}}}
\newcommand{\bcp}{\odot}

\newcommand{\Mod}[1]{\mathsf{Mod}_{{#1}}}
\newcommand{\Alg}[1]{\mathsf{Alg}_{{#1}}}
\newcommand{\Hom}{\mathrm{Hom}}
\newcommand{\Sch}{\Lambda^{\mathrm{Sch}}}
\newcommand{\Wsch}{W^{\mathrm{Sch}}}
\newcommand{\anti}[1]{{\{#1\}}}
\newcommand{\ser}{\sigma}
\newcommand{\sergh}{\sigma_{\mathrm{gh}}}
\newcommand{\schurmap}{v}
\newcommand{\Nhat}[1]{\bN[{#1}]{\sphat}\,}

\newcommand{\CC}{\mathsf{C}}
\newcommand{\Aff}{\mathsf{Aff}}
\newcommand{\prt}{P}
\newcommand{\Forg}{\Lambda^{{\mathrm{for}}}}

\newcommand{\smcoprod}{{\,\scriptstyle\amalg\,}}
\newcommand{\id}{{\mathrm{id}}}
\newcommand{\gp}{\mathfrak{p}}
\newcommand{\lttn}[1]{\Lambda_{\bN,(p),{#1}}}
\newcommand{\displaylabelfork}[6]{{	\entrymodifiers={+!!<0pt,\fontdimen22\textfont2>}
	\def\objectstyle{\displaystyle}
\xymatrix{{#1} \ar^-{#2}[r] & {#3} \ar@<0.7ex>^-{#4}[r]\ar@<-0.7ex>_-{#5}[r] & {#6}}}}
\newcommand{\displayfork}[3]{{\displaylabelfork{#1}{}{#2}{}{}{#3}}}

\newtheoremstyle{mythm}{}{}%
  {\itshape}%   Body font
  {}%           Indent amount (empty = no indent, \parindent = para indent)
  {\bfseries}%  Thm head font
  {}%           Punctuation after thm head
  { }%          Space after thm head (\newline = linebreak)
  {\thmnumber{#2.\hspace{1.5mm}}\thmname{#1}\thmnote{ {\mdseries(#3)}}.}%  Thm head spec

\newtheoremstyle{intro}{}{}%
  {\itshape}%	Body font
  {}%         	Indent amount (empty = no indent, \parindent = para indent)
  {\bfseries}% 	Thm head font
  {}%        	Punctuation after thm head
  { }%     	Space after thm head (\newline = linebreak)
  {\thmname{#1}\thmnumber{ #2}\thmnote{ #3}.}% Thm head spec

\newtheorem{question}{Question}
\numberwithin{equation}{subsection}
\theoremstyle{mythm}
\newtheorem{theorem}[subsection]{Theorem}
\newtheorem{proposition}[subsection]{Proposition}
\newtheorem{lemma}[subsection]{Lemma}
\newtheorem{corollary}[subsection]{Corollary}

\theoremstyle{intro}
\newtheorem*{thmintro}{Theorem}

\newcommand{\wis}[1]{{\text{\em \usefont{OT1}{cmtt}{m}{n} #1}}}
\def\Spec{\wis{Spec}}

\renewcommand{\theequation}
  {\arabic{section}.\arabic{subsection}.\arabic{equation}}

\makeindex

\title[Witt vectors, semirings, and total positivity]{Witt vectors, semirings, and total positivity}
\author[J.~Borger]{James Borger\thanks{Supported the Australian Research Council
under a Discovery Project (DP120103541) and a Future Fellowship (FT110100728).}}

\begin{document}
\setcounter{page}{273}

\begin{abstract} 
We extend the big and $p$-typical Witt vector functors from commutative rings to commutative
semirings. In the case of the big Witt vectors, this is a repackaging of some standard facts about
monomial and Schur positivity in the combinatorics of symmetric functions. In the $p$-typical case,
it uses positivity with respect to an apparently new basis of the $p$-typical symmetric functions.
We also give explicit descriptions of the big Witt vectors of the natural numbers and of the
nonnegative reals, the second of which is a restatement of Edrei's theorem on totally positive
power series. Finally we give some negative results on the relationship between truncated Witt vectors 
and $k$-Schur positivity, and we give ten open questions.
\end{abstract}



%\date{\today. \printtime}

\maketitle

\setcounter{tocdepth}{2}

\tableofcontents

\section*{Introduction}
Witt vector functors are certain functors from the category of rings (always commutative)
to itself. They come in different flavors, but each of the ones we will consider
sends a ring $A$
to a product $A\times A\times\cdots$ with certain addition and multiplication laws of an exotic
arithmetic nature. For example, for each prime $p$, there is the \emph{$p$-typical} Witt vector 
functor $W_{(p),n}$ of length $n\in \bN\cup\{\infty\}$.
As sets, we have $W_{(p),n}(A)=A^{n+1}$. When $n=1$, the
ring operations are defined as follows:
	\begin{align*}
		(a_0,a_1) + (b_0,b_1) &= (a_0+b_0,\ 
			a_1+b_1 -\sum_{i=1}^{p-1}\frac{1}{p}\binom{p}{i}a_0^ib_0^{p-i}) \\
		(a_0,a_1) (b_0,b_1) &= (a_0 b_0,\  a_0^p b_1 + a_1b_0^p + pa_1b_1) \\
		0 &= (0,0) \\
		1 &= (1,0).
	\end{align*}
For $n\geq 2$, the formulas in terms of coordinates are too complicated to be worth writing down.
Instead we will give some simple examples:
	\begin{align*}
		W_{(p),n}(\bZ/p\bZ) &\cong \bZ/p^{n+1}\bZ, \\
		W_{(p),n}(\bZ) &\cong 
			\bigsetof{\langle{x_0,\dots,x_n}\rangle\in\bZ^{n+1}}{x_i\equiv x_{i+1}\bmod p^{i+1}}.
	\end{align*}
In the second example, the ring operations are performed componentwise; in particular,
the coordinates are not the same as the coordinates above. For another example,
if $A$ is a $\bZ[1/p]$-algebra, then $W_{(p),n}(A)$ is isomorphic after a change of
coordinates to the usual product ring $A^{n+1}$.
This phenomenon holds for other kinds of Witt vectors---when all relevant primes 
are invertible in $A$, then the Witt vector ring of $A$ splits up as a product ring.

The $p$-typical Witt vector functors were discovered by Witt in 1937~\cite{Witt:Vectors} and have
since become a central construction in $p$-adic number theory, especially in $p$-adic Hodge theory,
such as in the construction of Fontaine's period rings~\cite{Fontaine:Corps-des-periodes-p-adiques}
and, via the de~Rham--Witt complex of Bloch--Deligne--Illusie, in crystalline
cohomology~\cite{Illusie:dRW}. Also notable, and related, is their role in the algebraic K-theory
of $p$-adic rings, following Hesselholt--Madsen~\cite{Hesselholt-Madsen:K-theory-of-local-fields}.

There is also the \emph{big} Witt vector 
functor. For most of this chapter, we will think about it from the point of view of symmetric
functions, but for the moment, what is important is that it is formed by composing
all the $p$-typical functors in an inverse limit:
	\begin{equation}
	\label{eq:big-W-factorization}
	W(A) = \lim_n W_{(p_1),\infty}(W_{(p_2),\infty}(\cdots (W_{(p_n),\infty}(A))\cdots)),
	\end{equation}
where $p_1,\dots,p_n$ are the first $n$ primes, and the transition maps are given by projections
$W_{(p),\infty}(A)\to A$ onto the first component. (A non-obvious fact is that the $p$-typical
functors commute with each other, up to canonical isomorphism; so in fact the ordering of the
primes here is unimportant.) This has an adelic flavor, and indeed it is possible to unify the
crystalline cohomologies for all primes using a (``relative'') de Rham--Witt complex for the big 
Witt vectors. However the cohomology of this complex is determined by the crystalline cohomologies 
and the comparison maps to algebraic de Rham cohomology; so the big de Rham--Witt cohomology
does not, on its own, provide any new information.

But because of this adelic flavor, it is natural to ask whether the infinite prime plays a role. The answer is
yes---big Witt vectors naturally accommodate positivity, which we will regard as $p$-adic
integrality at the infinite prime. On the other hand, we will have little to say about other aspects of the
infinite prime, such as archimedean norms or a Frobenius operator at $p=\infty$.

To explain this in more detail, we need to recall some basics of the theory of symmetric functions.
The big Witt vector functor is represented, as a set-valued 
functor, by the free polynomial ring $\Lambda_\bZ=\bZ[h_1,h_2,\dots]$. So as sets, we have
	$$
	W(A) = \Hom_{\Alg{\bZ}}(\Lambda_\bZ, A) = A\times A\times \cdots.
	$$
If we think of $\Lambda_\bZ$ as the usual ring of symmetric functions in infinitely many formal 
variables $x_1,x_2,\dots$ by writing
	$$
	h_n = \sum_{i_1\leq\dots\leq i_n} x_{i_1}\cdots x_{i_n},
	$$
then the ring operations on $W(A)$ are determined by two well-known coproducts
in the theory of symmetric functions. The addition law is determined by the coproduct
$\Delta^+$ on $\Lambda_\bZ$ for which the power-sum
symmetric functions $\psi_n=\sum_i x_i^n$ are primitive,
	$$
	\Delta^+(\psi_n) = \psi_n \tn 1 + 1\tn \psi_n,
	$$
and the multiplication law is determined 
by the coproduct for which the power sums are group-like,
	$$
	\Delta^\times(\psi_n) = (\psi_n\tn 1)(1\tn\psi_n) = \psi_n\tn\psi_n.
	$$
This is also a fruitful point of view for the $p$-typical functors: $W_{(p),\infty}$ is 
representable by the free subring $\Lambda_{\bZ,(p)}=\bZ[\theta_1,\theta_{p},\theta_{p^2},\dots]$ 
of \emph{$p$-typical} symmetric functions, where the $\theta_{n}\in\Lambda_\bZ$ are the Witt 
symmetric functions, which are defined recursively by the relations
	$$
	\psi_{n} = \sum_{d\mid n} d\theta_{d}^{n/d}.
	$$
The ring operations on the $p$-typical Witt functors are equivalent to coproducts on 
$\Lambda_{\bZ,(p)}$, and these are compatible with the two coproducts on $\Lambda_\bZ$.
In fact, $\Lambda_\bZ$ can be reconstructed from all the $\Lambda_{\bZ,(p)}$
as a kind of Euler product:
	\begin{equation}
		\label{eq:Euler-product}
	\Lambda_\bZ = \Lambda_{\bZ,(2)} \bcp \Lambda_{\bZ,(3)} \bcp \Lambda_{\bZ,(5)} \bcp \cdots,
	\end{equation}
where $\bcp$ is the operation that on representing objects corresponds to composition
on functors. This is in fact just another expression of formula~(\ref{eq:big-W-factorization}),
or alternatively of Wilkerson's theorem~\cite{Wilkerson}.

Now, the relation of all this to the infinite prime is that there is a well-known positivity
structure on $\Lambda_\bZ$. This is the subset $\Lambda_\bN$ consisting of symmetric functions that
have nonnegative coefficients when viewed as series in the formal variables $x_1,x_2,\dots$. It is
closed under addition and multiplication and contains $0$ and $1$; so it is a semiring, or more
plainly, an $\bN$-algebra. It is also well known that the coproducts $\Delta^+$ and $\Delta^\times$
above are induced by coproducts on $\Lambda_\bN$, and so one might hope to use them to extend the
big Witt construction to all $\bN$-algebras, and hence to incorporate some positivity information in
the usual theory of Witt vectors. This is indeed possible and the primary purpose of this chapter
is to write it all down in some detail. 

In fact, there is another such model over $\bN$, which is also well known.
It is the sub-$\bN$-algebra $\Sch$ consisting of symmetric functions which have
nonnegative coordinates with respect to the basis of Schur symmetric functions.

\begin{thmintro}[A]
	\label{thm:intro-big}
	The functors $\Hom_{\Alg{\bN}}(\Lambda_\bN,\vbl)$ and $\Hom_{\Alg{\bN}}(\Sch,\vbl)$ 
	extend the big Witt vector functor $W$ from $\bZ$-algebras to $\bN$-algebras. Each has a unique 
	comonad structure compatible with that on $W$.
\end{thmintro}

In terms of actual mathematical content, this is just a repackaging of some standard positivity
facts in the theory of symmetric functions. Thus a large part of this chapter is expository.
Its goal is to convince the reader that there is interesting mathematics where Witt vectors,
combinatorial positivity, and semiring theory meet. To this end, I have included a number of open
questions, which I hope will serve to focus readers' minds and stimulate their interest. Most of the
questions are precise, of the yes/no variety, and some are no doubt well within reach.

To give an example of something from this territory, I will report one new observation, which is
that there is also a positive model for the $p$-typical Witt vector functors.

\begin{thmintro}[B]
	\label{thm:intro-p-typical}
	There is a representable comonad on the category of $\bN$-algebras which agrees with the
	$p$-typical Witt vector functor $W_{(p),\infty}$ on $\bZ$-algebras.
\end{thmintro}

As with theorem A, the representing object is given by a positivity condition with respect to a
$\bZ$-basis, in this case of $\Lambda_{\bZ,(p)}$. Write $d_p=-\theta_p=(\psi_1^p-\psi_p)/p$. Consider the
(finite) monomials of the form
	$$
	\prod_{i,j\geq 0} (\psi_p^{\circ i}\circ d_p^{\circ j})^{m_{ij}},
	$$
where $\circ$ denotes the plethysm operation on $\Lambda_\bZ$, and where $m_{ij}<p$. Then this 
family of monomials is a $\bZ$-basis for $\Lambda_{\bZ,(p)}$. Its $\bN$-linear span 
$\Lambda_{\bN,(p)}$ is a sub-$\bN$-algebra of $\Lambda_\bZ$, and the functor on $\bN$-algebras it
represents admits a unique comonad structure compatible with that of $W_{(p),\infty}$. To my 
knowledge, this basis of $\Lambda_{\bZ,(p)}$ has not been considered before.

In one way, the theory around theorem B is more satisfactory than that around theorem A. This is
that it also works for the $p$-typical Witt vectors of finite length. I initially hoped that bases
of $k$-Schur functions of Lapointe--Lascoux--Morse (see the book~\cite{Lam-et-al:k-Schur-book})
would allow us to define $\bN$-algebras of big Witt vectors of finite length, but this turned out
to be false. See section~\ref{sec:k-Schur} for some details on this negative result.

There is also a larger purpose to this chapter, which is to show that the formalism of (commutative)
semirings---and more broadly, scheme theory over $\bN$---is a natural and well-behaved formalism,
both in general and in its applications to Witt vectors and positivity. It has gotten almost no
attention from people working with scheme theory over $\bZ$, but it deserves to be developed
seriously---and independently of any applications, which are inevitable in my view. Let me conclude
with some words on this.

Arithmetic algebraic geometry can be regarded as the study of systems of polynomials equations
over $\bZ$. Such a system is equivalent to a presentation of a $\bZ$-algebra; so one could say that
arithmetic algebraic geometry is the study of the category of $\bZ$-algebras. Of course, arithmetic
algebraic geometers study many other objects, such as nonaffine schemes over $\bZ$, line bundles
over them, and so on, but let us consider these as objects of derived interest, as tools for
understanding polynomial equations. In fact, many such concepts are formally inevitable once we allow
ourselves the category of $\bZ$-algebras and some general category theory.

Let me recall how this works for algebraic spaces. The category of affine schemes is defined to be
the opposite of the category of rings. It has a Grothendieck topology on it, the fppf topology,
where covers are given by fppf algebras, those that are faithfully flat and finitely presentable.
The category of all algebraic spaces over $\bZ$ (a slight enlargement of the category of schemes)
is what one gets by taking the closure of the category of affine schemes under adjoining arbitrary
coproducts and quotients by fppf equivalence relations.

This is a completely formal process. (For instance, see
To\"en--Vaqui\'e~\cite{Toen-Vaquie:algebrisation,Toen-Vaquie:Under-Spec-Z}.) Given a category
$\mathsf{C}$ that looks enough like the opposite of the category of rings and a well-behaved class of
equivalence relations, we can produce \emph{an} algebraic geometry from $\mathsf{C}$ by gluing
together objects of $\mathsf{C}$ using the given equivalence relations. In particular, we can do this
with the category of $\bN$-algebras and produce a category that could be called the category of {\em
schemes over $\bN$}\index{scheme over $\bN$}. This brings positivity into algebraic geometry at a very
basic level. In arithmetic algebraic geometry today, and specifically in global class field theory,
positivity is treated in an \emph{ad hoc} manner, much it seems as integrality was before the arrival
of scheme theory in the 1950s.

On the other hand, as it appears to me, most people working on semiring theory follow a tradition
close to general algebra or even computer science. Scheme theory has had little influence. As
someone raised in that tradition, I find this unacceptable. The category of rings is the same as
the category of semirings equipped with a map from $\bZ$. In other words, one might say that
arithmetic algebraic geometry is nothing more than semiring theory over $\bZ$. One would therefore
expect an active interest in finding models over $\bN$, or the nonnegative reals, of as
many of the objects of usual algebraic geometry over $\bZ$ as possible, just as one always tries to
find models for objects of classical algebraic geometry, such as moduli spaces, over $\bZ$. Yet
such an effort seems to be nearly nonexistent. Perhaps one reason for this is that most existing
expositions of scheme theory begin by considering spectra of prime ideals, and it is less clear how
to mimic this approach over $\bN$. Or perhaps people are more interested in designing foundations
for specific applications, such as tropical algebraic geometry, rather than developing general
tools. Whatever the case, it is important to get beyond this.

So in the first two sections, I give a category-theoretic account of the very basics of semiring
theory and algebraic geometry over $\bN$. It is largely expository. I hope it will demonstrate to
people who are skeptical that the basic constructions of scheme theory extend to $\bN$, and
demonstrate to semiring theorists a point of view on their subject that emphasizes macroscopic
ideas, such as flatness, base change, and descent, more than what is common. So at least the general formalism over $\bN$ can be brought up closer to that over $\bZ$.

If arithmetic algebraic geometry provides the motivation here and semiring theory provides the
formalism, then algebraic combinatorics provides us with the positivity results. These are needed to
define Witt vectors of semirings that do not contain $-1$; they could be viewed as the analogues at
the infinite prime of the slightly subtle $p$-adic congruences needed to define the $p$-typical Witt
vectors for rings that do not contain $1/p$. But I also hope that combinatorialists will find
something fresh in our emphasis on Witt vectors rather than symmetric functions. While the two are
equivalent, they often suggest different questions. For instance, the coproduct $\Delta^\times$ on
$\Lambda_\bZ$ has gotten much less attention than $\Delta^+$. But these are just the co-operations
that induce the multiplication and addition operations on Witt vectors. Although it is not without
interest to view Witt vectors only as abelian groups, the real richness of the theory and their role
in arithmetic algebraic geometry comes when we remember their full ring structure (or even better,
their $\Lambda$-ring structure). So to a specialist in Witt vectors, ignoring $\Delta^\times$
might feel like missing the whole point. Also, aspects of symmetric functions related to the finite
primes seem under-studied in the algebraic combinatorics community. For instance, the ring
$\Lambda_{\bZ,(p)}=\bZ[\theta_1,\theta_p,\dots]$ of $p$-typical symmetric functions is, it appears
to me, nearly unknown there. (The symmetric function $\theta_p$ does appear in
Macdonald~\cite{Macdonald:SF} as $-\varphi_p$ on p.\ 120.)

I would like to thank Thomas Lam and Luc Lapointe for some helpful correspondence, especially on $k$-Schur
functions. I would also like to thank Lance Gurney, Lars Hesselholt, and especially Darij Grinberg for making
some observations that I have included. All automated computation I did for this paper was done in the software
system Sage~\cite{Sage}, especially using the algebraic combinatorics features developed by the
\texttt{Sage-Combinat} community~\cite{Sage-Combinat}.

Finally, I should mention that Connes and Consani~\cite{Connes-Consani:char-one} have developed a theory of
Witt vectors for certain algebras over the Boolean algebra $\{0,1\}$. It would be interesting to know if there
is any relation with our theory.

\setcounter{section}{0}
\section*{Conventions}
The word \emph{positive}\index{positive} will mean $>0$, and \emph{negative}\index{negative} will mean $<0$. 
For any subset 
$A$ of the field $\bR$ of real numbers, we will
write
	$$
	A_+:=\setof{x\in A}{x\geq 0}.\index{$A_+$}
	$$
The set $\bN$ of natural numbers is $\{0,1,2,\dots\}$. The ring $\bZ_p$ of $p$-adic integers
is $\lim_n \bZ/p^n\bZ$, and the field $\bQ_p$ of $p$-adic numbers is $\bZ_p[1/p]$.
	
For any category $\CC$, we will write $\CC(X,Y)=\Hom_\CC(X,Y)$ for short.

\section{Commutative algebra over $\bN$, the general theory}

The primary purpose of this section and the next one is to collect the definitions and formal
results of commutative algebra and scheme theory over $\bN$ that we will need. The reader is
encouraged to skip them at first and refer back only when necessary.

A general reference to the commutative algebra is Golan's book~\cite{Golan:book1}. While everything
here is essentially the same, there are some small differences. For instance, I have preferred to
drop the prefix \emph{semi} wherever possible and do not want to assume $0\neq 1$. I have also
followed the categorical method and used its terminology, because it gives the development a
feeling of inevitability that I think is absent from the more element-centric approaches.

\subsection{The category of $\bN$-modules}
\label{subsec:N-modules}
The category $\Mod{\bN}$\index{$\Mod{\bN}$} of $\bN$-modules is by definition the category of commutative monoids,
which we typically write additively. Thus an $\bN$-module is a set $M$ with a commutative binary
operation $+=+_M$ and an identity element $0=0_M$, and a homomorphism $M\to P$ is function $f\:M\to P$
such that $f(0_M)=0_{P}$ and $f(x+_My)=f(x)+_Pf(y)$ for all $x,y\in M$. As usual, the identity element
is unique when it exists; so we will often leave to the reader the task of specifying it.

For example, $\bN$ itself is an $\bN$-module under usual addition. It represents the
identity functor on $\Mod{\bN}$ in an evident (and unique) way.

\subsection{Submodules and monomorphisms}
A subset $P$ of an $\bN$-module $M$ is said to be a {\em sub-$\bN$-module}\index{sub-$\bN$-module} if it admits an $\bN$-module
structure making the inclusion $P\to M$ a homomorphism. Because a map of $\bN$-modules is 
%{\em injective}\index{injective map (of $\bN$-modules)} 
injective
%%%
if and only if it is a monomorphism, we will usually identify monomorphisms and submodules.

The dual statement is false---there are nonsurjective epimorphisms, for instance the usual
inclusion $\bN\to\bZ$.

\subsection{Products, coproducts}
The category has all products: $\prod_{i\in I}M_i$ is the usual product
set with identity $(\dots,0,\dots)$ and componentwise addition
	$$
	(\dots,m_i,\dots)+(\dots,m'_i,\dots) \dfn (\dots,m_i+m'_i,\dots).
	$$
It also has all coproducts: $\bigoplus_{i\in I}M_i$ is the sub-$\bN$-module of $\prod_{i\in I}M_i$
consisting of the vectors $(m_i)_{i\in I}$ such that $m_i=0$ for all but finitely many $i\in I$.

In particular, we can construct free objects: given any set $S$ and any set map $f\:S\to M$, the
morphism $\bigoplus_{s\in S}\bN\to M$ defined by $(n_s)_{s\in S}\mapsto \sum_{s\in S}
n_s f(s)$ is the unique extension of $f$ to an $\bN$-module map $\bigoplus_{i\in S}\bN\to M$.

\subsection{Internal equivalence relations, quotients, and epimorphisms}
A subset $E\subseteq M\times M$ is said to be a {\em $\Mod{\bN}$-equivalence relation}\index{$\Mod{\bN}$-equivalence relation} if it is both an
equivalence relation on $M$ and a sub-$\bN$-module of $M\times M$.

Given any homomorphism $f\:M\to M'$ of $\bN$-modules, the induced equivalence relation
$M\times_{M'} M$ is clearly a $\Mod{\bN}$-equivalence relation. Conversely, given any
$\Mod{\bN}$-equivalence relation $E$ on $M$, the set $M/E$ of equivalence classes has a unique
$\bN$-module structure such that the projection $M\to M/E$ sending $x$ to the equivalence class
$[x]$ of $x$ is a homomorphism of $\bN$-modules. In other words, the rules $[x]+[y]=[x+y]$ and
$0=[0]$ give a well-defined $\bN$-module structure on $M/E$.

\subsection{Generators and relations}
Let $R$ be a subset of $M\times M$. The $\Mod{\bN}$-equivalence relation $E$ generated by $R$ is the
minimal $\Mod{\bN}$-equivalence relation on $M$ containing $R$. It exists because any intersection of
submodules is a submodule and any intersection of equivalence relations is an equivalence relation. It
can be constructed explicitly by taking the transitive closure of the sub-$\bN$-module of $M\times M$
generated by $R$, the transpose of $R$, and the diagonal. Note the contrast with the theory of modules
over a ring, where taking the transitive closure is unnecessary. This gives the theory of modules over $\bN$ a dynamical feel which is absent when over a ring.

We can construct $\bN$-modules in terms of generators and relations by combining this with free
construction above. Clearly, much of this works in much more general categories, especially
categories of algebras, as defined below. We will use such generalizations without comment.
In the present case and others, we will often write $M/(\dots,m_i=n_i,\dots)$ for $M/E$ if
$R=\{\dots,(m_i,n_i),\dots\}$.

\subsection{$\Hom$ and $\tn$}
The set $\Hom(M,P)=\Mod{\bN}(M,P)$ of $\bN$-module homomorphisms $M\to P$ is itself an 
$\bN$-module under pointwise addition:
	\begin{equation}
	\label{eq:N-mod-str-on-Hom}
	(f+g)(x) := f(x)+_Pg(x), \quad\quad 0(x) := 0_P.
	\end{equation}
We will also use the notation $\Hom_{\bN}(M,P)$. It is functorial in $M$ (contravariant) and $P$
(covariant). For any fixed $\bN$-module $M$, the functor $\Hom(M,\vbl)$ has a left adjoint, which
we write $M\tn \vbl$, or $M\tn_{\bN} \vbl$ for clarity. In other words, $M\tn M'$ is characterized
by the property that a homomorphism $M\tn M'\to P$ is the same as set map $\langle
\vbl,\vbl\rangle\:M\times M' \to P$ which is \emph{$\bN$-bilinear}\index{bilinear morphism} in that 
it is an
$\bN$-module map in each argument if the other argument is fixed.
It follows that if we denote the image of an element $(m,m')$ under the universal bilinear map
$M\times M'\to M\tn M'$ by $m\tn m'$, then $M\tn M'$ is the commutative monoid generated by symbols
$m\tn m'$, for all $m\in M$ and $m'\in M'$, modulo all relations of the form
	\begin{alignat*}{2}
	(m_1+_M m_2)\tn m'&= m_1\tn m' + m_2\tn m',\quad\quad & 0_M\tn m' &= 0, \\
	m\tn(m'_1+_{M'}m'_2) &= m\tn m'_1 + m\tn m'_2,\quad\quad & m\tn 0_{M'} &= 0.
	\end{alignat*}
Then $\tn$ makes $\Mod{\bN}$ into a symmetric monoidal category with identity $\bN$.	

\subsection{$\bN$-algebras}
\label{subsec:N-algebras}
An {\em $\bN$-algebra}\index{$\bN$-algebra} (soon to be understood to be commutative) is defined to be a monoid in $\Mod{\bN}$
with respect to the monoidal structure $\tn$. Thus an $\bN$-algebra is a set $A$ with two
associative binary operations $+,\times$ and respective identity elements $0,1$ such that $+$ is
commutative and $\times$ distributes over $+$ and satisfies $0\times x=x\times 0=0$. We usually
write $xy=x\times y$. We will sometimes use the term \emph{semiring}\index{semiring} as a synonym 
for $\bN$-algebra.

A {\em morphism}\index{morphism!of $\bN$-algebras} $A\to B$ of $\bN$-algebras is a morphism of monoids in the monoidal category $\Mod{\bN}$.
In other words, it is a function $f\:A\to B$ which satisfies the identities 
	$$
	f(0)=0,\quad f(x+y)=f(x)+f(y),\quad f(1)=1,\quad f(xy)=f(x)f(y).
	$$
The category formed by $\bN$-algebras and their morphisms is denoted $\Alg{\bN}$\index{$\Alg{\bN}$}.

For example, the $\bN$-module $\bN$ admits a unique $\bN$-algebra structure; multiplication is usual
integer multiplication. It is the initial object in $\Alg{\bN}$. Likewise, ${0}$ with its unique
$\bN$-algebra structure is the terminal object. For any subring $A\subseteq \bR$, the subset
$A_+:=\setof{x\in A}{x\geq 0}$ is a sub-$\bN$-algebra of $A$.

The category of rings is equivalent in an evident way to the full subcategory of $\Alg{\bN}$
spanned by objects in which $1$ has an additive inverse.

\subsection{Commutativity assumption}
From now on in this paper, all $\bN$-algebras will be understood to be commutative under $\times$ unless stated otherwise. However for much of the rest of this section, this is just for convenience.

Also for the rest of this section, $A$ will denote an $\bN$-algebra.

\subsection{$A$-modules and $A$-algebras}
One defines {\em $A$-modules} and {\em $A$-algebras} in the obvious way.
An {\em $A$-module}\index{$A$-module} (also called an {\em $A$-semimodule} in the literature)
is an $\bN$-module equipped with an action of 
$A$ with respect to the monoidal structure $\tn$.
So it is an $\bN$-module $M$ equipped with an $\bN$-module map $A\tn M\to M$, written
$a\tn m\mapsto am$, such that the following identities are satisfied:
	$$
	1m = m, \quad\quad (ab)m=a(bm).
	$$
A {\em morphism}\index{morphism!of $A$-modules} of $A$-modules $M\to P$ is an $\bN$-linear map $f\:M\to P$ satisfying the identity
$f(am)=af(m)$. The category of $A$-modules is denoted $\Mod{A}$\index{$\Mod{A}$}.
We will sometimes write $\Hom_A(M,P)=\Mod{A}(M,P)$ for the set of $A$-module morphisms.
Observe when $A=\bN$, the category just defined agrees with that defined 
in~(\ref{subsec:N-modules}).

An {\em $A$-algebra}\index{$A$-algebra} is an $\bN$-algebra $B$ equipped with a morphism $i_B\:A\to B$. A {\em morphism}\index{morphism!of $A$-algebras} $B\to C$
of $A$-algebras is a morphism $f\:B\to C$ of $\bN$-algebras such that $f\circ i_B = i_C$. The
category of $A$-algebras is denoted $\Alg{A}$\index{$\Alg{A}$}. As with modules, when $A=\bN$, the category of
$\bN$-algebras as defined here agrees with that defined in~(\ref{subsec:N-algebras}).

Also observe that if $A$ is a ring, then these definitions of $A$-module and $A$-algebra agree
with the usual ones in commutative algebra. In particular, a $\bZ$-module is the same as an
abelian group, and a $\bZ$-algebra is the same as a ring.

\subsection{$\Hom_A$ and $\tn_A$}
The set $\Hom_A(M,P)$ has a natural $A$-module structure given by pointwise operations. In other
words, it is a sub-$\bN$-module of $\Hom_{\bN}(M,P)$, and its $A$-module structure is given by the
identity $(af)(x)=a f(x)$. Of course, this uses the commutativity of multiplication on $A$. 

For any fixed $A$-module $M$, the functor $\Hom_A(M,\vbl)$ has a left adjoint, which we write
$M\tn_A \vbl$:
	$$
	\Hom_A(M\tn_A M',N) = \Hom_A(M',\Hom_A(M,N)).
	$$
(Again, when $A=\bN$ this agrees with the functor $M\tn\vbl$ defined above.)
As above, an {\em $A$-linear map}\index{$A$-linear map} $M\tn_A M'\to P$ is the same a set map
$\langle\vbl,\vbl\rangle\:M\times M'\to P$ which is \emph{$A$-bilinear}\index{$A$-bilinear map} in 
the sense that it is an $A$-module map in each argument when the other argument is held fixed.
Thus $M\tn_A M'$ equals the quotient of $M\tn_{\bN}M'$  by all relations of the form 
	\begin{equation}
		(am)\tn m'=m\tn(am'),
	\end{equation}
with its $A$-module structure given by $a(m\tn m')=am\tn m' = m\tn am'$.

\subsection{Limits and colimits of $A$-modules and $A$-algebras}
The category $\Mod{A}$ has all limits and colimits. Limits, coproducts, and filtered colimits can be
constructed as when $A$ is a ring, but coequalizers might be less familiar. Given a pair of maps
$f,g\:M\to P$ in $\Mod{A}$, the {\em coequalizer}\index{coequalizer} can be constructed as the
quotient of $P$ by the $\Mod{A}$-equivalence relation generated by the subset
$\setof{(f(x),g(x))}{x\in M}\subseteq P\times P$.

\subsection{Warning: kernels and cokernels}

There are reasonable notions of kernel and cokernel, but we will not need them. The {\em kernel}\index{kernel} of a map
$f\:M\to N$ of $A$-modules is defined to be the pull-back $M\times_N (0) \to M$, and the cokernel
is the push-out $N\to N\oplus_M (0)$. Many familiar properties of kernels and cokernels from
abelian categories fail to hold for modules over semirings. For instance the sum map
$\bN\oplus\bN\to\bN$ has trivial kernel and cokernel, but it is not an isomorphism. So kernels and
cokernels play a less prominent role here than they do in abelian categories. Equalizers and
coequalizers are more useful.

\subsection{Base change, induced and co-induced modules}
Let $B$ be an $A$-algebra, and let $M$ be an $A$-module. Then $B\tn_A M$ and $\Hom_{A}(B,M)$ are
$B$-modules in the evident ways. These constructions give the left and right adjoints of the
forgetful functor $U$ from $B$-modules to $A$-modules. They are called the $B$-modules
\emph{induced}\index{induced module} and \emph{coinduced}\index{coinduced module} by $A$. It is also clear that the forgetful functor $U$ is both
monadic and comonadic.

\subsection{Limits and colimits of $A$-algebras}
Like $\Mod{A}$, the category $\Alg{A}$ also has all limits and colimits.
Limits, coproducts, and filtered colimits can again be constructed as when
$A$ is a ring. In particular, coproducts are tensor products:
	$$
	B \smcoprod C = B\tn_A C.
	$$
More generally, the {\em coproduct}\index{coproduct} of any family $(B_i)_{i\in I}$ is the tensor product of all $B_i$
over $A$. And as with modules, the {\em coequalizer}\index{coequalizer} of two $A$-algebra morphisms $f,g\:
B\rightrightarrows C$ is $C/R$, where $R$ is the equivalence relation internal to $\Alg{A}$ on $C$
generated by the relation $\setof{(f(x),g(x))\in C\times C}{x\in B}$.

\subsection{Base change for algebras}
For any $A$-algebra $B$, the forgetful functor $\Alg{B}\to\Alg{A}$ has a left adjoint. It sends
$C$ to $B\tn_A C$, where the $B$-algebra structure is given by the map $b\mapsto b\tn 1$.

\subsection{Flat modules and algebras}
Let $M$ be an $A$-module.
Because the functor $M\tn_A\vbl\:\Mod{A}\to\Mod{A}$ has a right adjoint, it preserves all colimits.
Since finite products and coproducts agree, it also preserves finite products. If it preserves
equalizers, we say $M$ is \emph{flat}\index{flat!module}. In this case, $M\tn_A\vbl$ preserves all finite limits.

Observe that while not all monomorphisms are equalizers, it is nevertheless true that tensoring 
with a flat module preserves monomorphisms. Indeed $f\:N\to P$ is a monomorphism if and only
if the diagonal map $N\to N\times_P N$ is an isomorphism, and this property is
preserved by tensoring with a flat module.

Flatness is preserved under base change. An $A$-algebra is said to be {\em
flat}\index{flat!$A$-algebra} if it is flat when regarded as an $A$-module. If $A\to B$ and $B\to C$
are flat, then so is the composition $A\to C$. More generally, if $B$ is a flat $A$-algebra, and $M$
is a flat $B$-module, then $M$ is flat as an $A$-module.

\subsection{Examples of flat modules}
Any free module is flat. Any filtered colimit of flat modules is flat. We will see
in~(\ref{pro:flatness-is-local}) below that flatness is a flat-local property. So for example a
module is flat if it is flat-locally free.

It is a theorem of Govorov and Lazard that any flat module over a ring can be represented as a
filtered colimit of free modules. This continues to hold for modules over any $\bN$-algebra. As
over rings, this is tantamount to an equational criterion for flatness for modules over
$\bN$-algebras, but now we must consider all relations of the form $\sum_i a_ix_i = \sum_i b_i x_i$
instead of just those of the form $\sum_i a_i x_i=0$, as one usually does over rings. See
Katsov~\cite{Katsov:flat-semimodules}.

If $S$ is a multiplicative subset of $A$, let $A[1/S]$\index{$A[1/S]$} denote the initial $A$-algebra in which
every element of $S$ becomes multiplicatively invertible. Then $A[1/S]$ is flat because it can be
represented as $\colim_{s\in S}A$, where for all $s,t\in S$ there is a transition map $A\to A$ from
position $s$ to position $st$ given by multiplication by $t$.

But it is completely different if we adjoin additive rather than multiplicative inverses. We will
see in~(\ref{cor:flat-over-zerosumfree}) below that $\bZ$ is not a flat $\bN$-module. In fact, $0$
is the only $\bZ$-module that is flat over $\bN$. It is also the only $\bR$-vector space that is
flat over $\bRpl$.

\section{The flat topology over $\bN$}

The purpose of this section is to give some idea of scheme theory over $\bN$. It is the point of
view I prefer for the mathematics of this chapter, but I will not use it in a serious way.

Scheme theory and the flat topology over $\bN$ were apparently first considered in
To\"en--Vaqui\'e~\cite{Toen-Vaquie:Under-Spec-Z}. Lorscheid has considered a different but related
approach~\cite{Lorscheid:Blueprints} (or see his chapter in this volume). In recent years, positivity
structures in algebraic geometry have appeared in some interesting applications, although in an
\emph{ad hoc} way. For example, let us mention the work of
Lusztig~\cite{Lusztig:survey-total-positivity},
Fock--Goncharov~\cite{Fock-Goncharov:higher-Teichmueller}, and
Rietsch~\cite{Rietsch:Toeplitz-matrices}.

\subsection{Flat covers}
Let us say that a family $(B_i)_{i\in I}$ of flat $A$-algebras is 
\emph{faithful}\index{faithful family (of flat algebras)}
if the family of base change functors
	\begin{equation}
		\label{eq:fpqc-base-change-functor}
	\Mod{A}\longmap \prod_{i\in I}\Mod{B_i}, \quad M\mapsto (B_i\tn_A M)_{i\in I}
	\end{equation}
reflects isomorphisms, that is, a map $M\to N$ of $A$-modules is an isomorphism if (and only
if) for every $i\in I$ the induced map $B_i\tn_AM \to B_i\tn_A N$ is. Let us say that a family of
flat $A$-algebras is an \emph{fpqc cover}\index{fpqc!cover} if it has a finite subfamily that is faithful.

\subsection{The fpqc topology}
For any $\bN$-algebra $K$, let $\Aff_K$\index{$\Aff_K$} denote the opposite of the category of $K$-algebras. 
For any $K$-algebra $A$, write $\Spec(A)$ for the corresponding object in $\Aff_K$. The fpqc covers
form a pretopology on $\Aff_K$, in the usual way. See To\"en--Vaqui\'e, proposition
2.4~\cite{Toen-Vaquie:Under-Spec-Z}. The resulting topology is called the \emph{fpqc topology}\index{fpqc!topology} or,
less formally, the \emph{flat topology}\index{flat!topology}.

One might also like to define an \emph{fppf topology}\index{fppf topology} topology by requiring
that each $B_i$ be finitely presented as an $A$-algebra. The following question is then natural:

\begin{question}
	Let $(B_i)_{i\in I}$ be a faithful family of flat $A$-algebras. If each $B_i$ is finitely
	presented as an $A$-algebra, is there a finite faithful subfamily?
\end{question}

When $K$ is $\bZ$, it is a fundamental fact from scheme theory that the answer is yes. To prove it
one combines quasi-compactness in the Zariski topology with the fact that faithfully flat morphisms
of finite presentation have open image.

\subsection{Faithfully flat descent}
Let $(B_i)_{i\in I}$ be a faithful family of flat $A$-algebras. Then the family of base change
functors~(\ref{eq:fpqc-base-change-functor}) is comonadic. As usual, this is just an application of
Beck's theorem in category theory. (See Borceux~\cite{Borceux:Handbook.v2}, theorem 4.4.4, p.\ 212.
See also theorem 2.5 of To\"en--Vaqui\'e~\cite{Toen-Vaquie:Under-Spec-Z}.) Thus the fibered
category of modules satisfies effective descent in the comonadic sense. If the family $(B_i)_{i\in
I}$ is finite, or more generally an fpqc cover, then the comonadic approach to descent agrees with
the Grothendieck's original one. So in either sense, in the fpqc topology, the fibered category of
modules satisfies effective descent, or it is a stack.

Thus descent allows us to recover the category of $A$-modules from that of modules over the cover.
As usual, this allows us to recover $A$ itself:

\begin{proposition}
	\label{pro:algebra-determined-by-module-category}
	Let $A$ be an $\bN$-module, and let $E$ denote the $\bN$-algebra (possibly noncommutative, a 
	priori) of natural endomorphisms of the identity functor on $\Mod{A}$. Then the canonical
	map $A\to E$ is an isomorphism.
\end{proposition}
\begin{proof}
	Let $\varphi$ be such a natural endomorphism. Set $a=\varphi_A(1)$. Then
	for any $M\in\Mod{A}$, the map $\varphi_M\:M\to M$ is multiplication by $a$. Indeed,
	for any $m\in M$, consider the map $f\:A\to M$ determined by $f(1)=m$. Then we have
	$\varphi_M(m)=\varphi_M(f(1)) = f(\varphi_A(1)) = f(a) = am$.
\end{proof}

\begin{proposition}
	\label{pro:flatness-is-local}
	Let $B$ be an $\bN$-algebra, and let $(C_i)_{i\in I}$ be a faithful family of flat $B$-algebras.
	\begin{enumerate}
		\item For any finite diagram $(M_j)_{j\in J}$ of $B$-modules, a map
			$M\to\lim_j M_j$ is an isomorphism if and only if each
			map $C_i\tn_B M\to\lim_j C_i\tn_B M_j$ is an isomorphism.
		\item Suppose $B$ is an $A$-algebra, for some given $\bN$-algebra $A$. Then $B$ is flat over $A$
			 if and only if each $C_i$ is.
		\item A $B$-module $N$ is flat if and only if each $C_i\tn_B N$ is flat as a $C_i$-module.
	\end{enumerate}
\end{proposition}
\begin{proof}
(1): First, because the family $(C_i)_{i\in I}$ is faithful, the map $M\to\lim_j M_j$ is an isomorphism if and only if each map $C_i\tn_B M\to C_i\tn_B \lim_j M_j$ is. Second, because each $C_i$ is flat, we have 
$C_i\tn_B\lim_j M_j = \lim_j (C_i\tn_B M_j)$. Combining these two statements proves (1).

(2): Consider a finite diagram $(M_j)_{j\in J}$ of $A$-modules. Suppose each $C_i$ is flat over $A$. 
Then the induced maps $C_i\tn_A \lim_j M_j \to \lim_j (C_i\tn_A M_j)$ are isomorphisms, and hence
so are the maps $$C_i\tn_B B\tn_A \lim_j M_j \to \lim_j (C_i\tn_B B \tn_A M_j).$$
Therefore by part (1), the map $B\tn_A \lim_j M_j \to \lim_j(B\tn_A M_j)$ is an isomorphism, and hence
$B$ is flat over $A$.

The converse holds because flatness is stable under composition.

(3): Suppose each $C_i\tn_B N$ is flat over $C_i$, and hence over $B$.
Then for any finite diagram $(M_j)_{j\in J}$ of $B$-modules, the maps 
$$C_i\tn_B N\tn_B \lim_j M_j\to \lim_j (C_i\tn_B N\tn_B M_j)$$ are isomorphisms. Therefore by part (1), the map
$N\tn_B\lim_j M_j\to \lim_j (N\tn_B M_j)$ is an isomorphism, and so $N$ is flat over $B$.

The converse holds because flatness is stable under base change.
\end{proof}

\subsection{Algebraic geometry over $\bN$}
We can then define the basic objects of algebraic geometry over any $\bN$-algebra $K$ in a formal
way, as in To\"en--Vaqui\'e~\cite{Toen-Vaquie:Under-Spec-Z}. A map $f\:\Spec(B) \to\Spec(A)$ is
\emph{Zariski open}\index{Zariski open morphism}
if the corresponding map $A\to B$ is a flat epimorphism of
finite presentation. One then defines {\em $K$-schemes}\index{$K$-scheme} by gluing together {\em affine $K$-schemes}\index{affine $K$-scheme} along
Zariski open maps. In To\"en--Vaqui\'e, all this takes place in the category of sheaves of sets on
$\Aff_K$ in the Zariski topology.

Presumably one could define a category of algebraic spaces over $\bN$ by adjoining quotients of
fppf or \'etale equivalence relations, under some meaning of these terms, as for example in
To\"en--Vaqui\'e, section 2.2~\cite{Toen-Vaquie:algebrisation}. But there are subtleties whether one uses
fppf maps, as in the question above, or \'etale maps, where it is not clear that the first
definition that comes to mind is the best one. So some care seems to be needed before we can have
complete confidence in the definition. In any case, we will certainly not need this generality.

\subsection{Extending fibered categories to nonaffine schemes}
Because we have faithfully flat descent for modules, we can define the category $\Mod{X}$ of
$X$-modules (i.e., quasi-coherent sheaves) over any $K$-scheme $X$. More generally, any fibered
category over $\Aff_K$ for which we have effective descent extends uniquely (up to some appropriate
notion of equivalence) to such a fibered category over the category of $K$-schemes. For example,
flat modules have this property by~(\ref{pro:flatness-is-local})(3). So we can make sense of a flat
module over any $\bN$-scheme.

Similarly, using~(\ref{pro:flatness-is-local})(2), there is a unique way of defining flatness for
morphisms $X\to Y$ of $K$-schemes that is stable under base change on $Y$ and fpqc-local on $X$.

We will conclude this section with some examples of flat-local constructions 
and properties.

\subsection{Additively idempotent elements}

An element $m$ of an $A$-module $M$ is \emph{additively idempotent}\index{additively idempotent 
element} if $2m=m$. The set $I(M)$
of such elements is therefore the equalizer
	$$
	\displaylabelfork{I(M)}{}{M}{x\mapsto x}{x\mapsto 2x}{M.}
	$$
Thus $I$ has a flat-local nature. Indeed, for any flat $A$-algebra $B$, the induced map
$B\tn_A I(M)\to I(B\tn_A M)$ is an isomorphism of $B$-modules, and
if $(B_i)_{i\in J}$ is an fpqc cover of $A$, then the induced map
	$$
	\displayfork{I(M)}{\prod_i I(B_i\tn_A M)}{\prod_{j,j'}I(B_j\tn_A B_{j'}\tn_A M)}
	$$
is an equalizer diagram.
It follows that, given an $\bN$-scheme $X$, this defines an $X$-module $I(M)$ for any $X$-module 
$M$. Thus the functor $I$ prolongs to a morphism of fibered categories $I\:\Mod{X}\to\Mod{X}$.

\subsection{Cancellative modules}
\label{subsec:cancellative-modules}

An $A$-module $M$ is \emph{additively cancellative}\index{cancellative module!additively} if the 
implication
	$$
	x+y=x+z \quad\Rightarrow\quad y=z
	$$
holds in $M$. This is equivalent to the following being an equalizer diagram:
	$$
%%%	\entrymodifiers={+!!<0pt,\fontdimen22\textfont2>}
%%%	\def\objectstyle{\displaystyle}
	\xymatrix@C=60pt{{M^2} \ar^-{(x,y)\mapsto (x,y,y)}[r] & {M^3} 
		\ar@<0.7ex>^-{(x,y,z)\mapsto x+y}[r]\ar@<-0.7ex>_-{(x,y,z)\mapsto x+z}[r] & {M.}}
	$$
Therefore, by~(\ref{pro:flatness-is-local}), additive cancellativity is a flat-local property.

Multiplicative cancellation works similarly. For any $a\in A$, let us say that $M$ is
\emph{$a$-cancellative}\index{cancellative module!$a$-cancellative} if the implication $$ax=ay \quad\Rightarrow\quad x=y$$ holds in $M$.
This is equivalent to the following being an equalizer diagram:
	$$
%%%	\entrymodifiers={+!!<0pt,\fontdimen22\textfont2>}
%%%	\def\objectstyle{\displaystyle}
	\xymatrix@C=50pt{{M} \ar^-{x\mapsto (x,x)}[r] & {M^2} 
		\ar@<0.7ex>^-{(x,y)\mapsto ax}[r]\ar@<-0.7ex>_-{(x,y)\mapsto ay}[r] & {M}}
	$$
Then $a$-cancellativity is also a flat-local property.

\subsection{Strong and subtractive morphisms}
A morphism $M\to N$ of $A$-modules induces two diagrams:
	$$
	\xymatrix{
	M\times M \ar^-{+}[r]\ar[d] & M\ar[d] & & M\times M \ar^-{+}[r]\ar[d] & M\ar[d] \\
	N\times N \ar^-{+}[r] & N & & M\times N \ar^-{+}[r] & N.
	}
	$$
We say it is \emph{strong}\index{strong morphism} if the first diagram
is Cartesian, and \emph{subtractive}\index{subtractive morphism}
if the second is. So a submodule $M\subseteq N$ is subtractive if and only if it closed under differences that
exist in $N$. Both properties are flat-local.

\subsection{Additively invertible elements}
For any $\bN$-algebra $A$ and any $A$-module $M$, consider the subset of additively invertible elements:
	\begin{equation}
		V(M):=\setof{x\in M}{\exists y\in M,\ x+y=0}.		
	\end{equation}
Then $V(M)$ is a sub-$A$-module of $M$.
The resulting functor is the right adjoint of the forgetful functor $\Mod{\bZ\tn_\bN A}\to\Mod{A}$,
and so we have
	$$
	V(M)=\Hom_A(\bZ\tn_\bN A,M).
	$$ 
It can also be expressed as an equalizer:
	\begin{equation}
%%%		\entrymodifiers={+!!<0pt,\fontdimen22\textfont2>}
%%%		\def\objectstyle{\displaystyle}
		\xymatrix@C=40pt{
		{V(M)} \ar^-{x\mapsto (x,-x)}[r] 
			& {M^2} \ar@<0.7ex>^-{0}[r]\ar@<-0.7ex>_-{+}[r] 
			& M
		}
	\end{equation}
Therefore $V$ has a local nature, and so we can define an $X$-module $V(M)$ for any module $M$
over any $\bN$-scheme $X$. In fact, since $V(M)$ is a group under addition, $V(M)$ is an
$(X\times\Spec(\bZ))$-module, and so $V$ can be viewed as a morphism $\Mod{X}\to
f_*(\Mod{X\times\Spec(\bZ)})$ of fibered categories, where $f$ denotes the canonical projection
$X\times\Spec(\bZ)\to X$.

Golan~\cite{Golan:book1} says $M$ is \emph{zerosumfree}\index{zerosumfree module} if $V(M)=0$, or
equivalently if $(0)$ is a strong submodule. The remarks above imply that being zerosumfree is a
flat-local property.

\begin{corollary}
	\label{cor:flat-over-zerosumfree}
	Let $A$ be a zerosumfree $\bN$-algebra, and let $M$ be a flat $A$-algebra.
	Then $M$ is zerosumfree. In particular, the zero module $(0)$ is the only flat $A$-module which is a 
	group under addition, and the map $A\to\bZ\tn_\bN A$ is flat if and only if 
	$\bZ\tn_\bN A=0$. 
\end{corollary}
\begin{proof}
	Since $M$ is flat, we have $V(M)=M\tn_A V(A)=M\tn_A 0=0$. 
\end{proof}

It follows formally that the only $\bZ$-scheme that is flat over a given zerosumfree scheme is the empty
scheme. Thus we would expect some subtleties in the spirit of derived functors when passing from algebraic
geometry over $\bN$ to that over $\bZ$, or from $\bRpl$ to $\bR$.

\section{Plethystic algebra for $\bN$-algebras}
\label{sec:plethystic-algebra}
Let $K,K',L$ be $\bN$-algebras.

\subsection{Models of co-$\CC$ objects in $\Alg{K}$}
\label{subsec:models-of-comodules}
Let $\CC$ be a category of the kind considered in universal algebra. Thus an object of $\CC$
is a set with a family of multinary operations satisfying some universal identities. For example,
$\CC$ could be the categories of groups, monoids, $L$-algebras, $L$-modules, Lie algebras over 
$L$, loops, heaps, and so on.

Let us say that a covariant functor $\Alg{K}\to\CC$ is {\em representable}\index{representable functor} if its underlying set-valued
functor is. Let us call the object representing such a functor a \emph{co-$\CC$
object}\index{co-$\CC$ object} in $\Alg{K}$. For example, a co-group structure on a $K$-algebra $A$
is the same as a group scheme structure on $\Spec(A)$ over $\Spec(K)$. Likewise, a co-$L$-algebra
structure on $A$ is the same as a $L$-algebra scheme structure on $\Spec(A)$ over $\Spec(K)$.

\subsection{Co-$L$-algebra objects in $\Alg{K}$}
Unpacking this further, we see
a co-$L$-algebra object of $\Alg{K}$ is a $K$-algebra $P$ together with $K$-algebra maps
	\begin{align}
		\begin{split}
		\Delta^+\: P&\longmap P\tn_K P  \\
		\Delta^{\times}\:P&\longmap P\tn_K P
		\end{split}
	\end{align}
subject to the condition that for all $A\in\Alg{K}$, the set $\Hom_K(P,A)$ equipped with the binary 
operations $+$ and $\times$ induced by $\Delta^+$ and $\Delta^{\times}$ is an $\bN$-algebra,	
plus an $\bN$-algebra map
	\begin{equation}
		\beta\: L\longmap\Alg{K}(P,K).
	\end{equation}
These properties can of course be expressed in terms of $P$ itself, without quantifying over any
variable algebras, as above with $A$. For example, $\Delta^+$ and $\Delta^\times$ must be
cocommutative and coassociative, and $\Delta^\times$ should codistribute over $\Delta^+$, and so on.

Similarly, the (unique) elements $0$ and $1$ in the $\bN$-algebras $\Hom_K(P,A)$
correspond to $K$-algebra morphisms
	\begin{equation}
	\label{eq:co-units}
		\eps^+\: P\to K, \quad\quad \eps^\times\:P\to K.
	\end{equation}
We will need these later.

We will often use the term \emph{$K$-$L$-bialgebra}\index{$K$-$L$-bialgebra} instead of 
co-$L$-algebra object of
$\Alg{K}$. This is not meant to suggest any relation to the usual meaning of the term
\emph{bialgebra} in the theory of Hopf algebras.
(Every $K$-$L$-bialgebra has two bialgebra structures in the usual
sense---$\Delta^+$ and $\Delta^\times$---but this is just a coincidence of terminology.)

Let $\Alg{K,L}$\index{$\Alg{K,L}$} denote the category of $K$-$L$-bialgebras. A {\em morphism}\index{morphism!of $K$-$L$-bialgebras} $P\to P'$ of $\Alg{K,L}$ is a
$K$-algebra map compatible with the co-operations $\Delta^+$, $\Delta^\times$, and $\beta$. In
other words, the induced natural transformation of set-valued functors must prolong to a natural
transformation of $L$-algebra-valued functors.

\subsection{Plethystic algebra} 
\label{subsec:plethystic-algebra}

In the case where $K$ and $L$ are $\bZ$-algebras, the theory of $K$-$L$-bialgebras was initiated in
Tall--Wraith~\cite{Tall-Wraith} and developed further in Borger--Wieland~\cite{Borger-Wieland:PA}.
It is clear how to extend the general theory developed there to $\bN$-algebras. In almost all
cases, the relevant words from \cite{Borger-Wieland:PA} work as written; at some places, obvious
changes are needed. The reader can also refer to
Stacey--Whitehouse~\cite{Stacey-Whitehouse:TW-monoids}, where the general theory is written down
for general universal-algebraic categories. (Also, see
Bergman--Hausknecht~\cite{Bergman-Hausknecht} for many fascinating case studies taken from
different categories, such as Lie algebras, monoids, groups, possibly noncommutative rings, and many
more.)

Let us list some of the main ideas we will need.
\begin{enumerate}
	\item The {\em composition product}\index{composition!product} is a functor 
		$\vbl\bcp_{L}\vbl\:\Alg{K,L}\times\Alg{L}\to\Alg{K}$. It is characterized by the
		adjunction
			$$
			\Alg{K}(P\bcp_L A,B) = \Alg{L}(A,\Alg{K}(P,B)).
			$$
	\item It has an extension to a functor 
		$\vbl\bcp_{L}\vbl\:\Alg{K,L}\times\Alg{L,K'}\to\Alg{K,K'}$
	\item This gives a monoidal structure (not generally symmetric) on the category $\Alg{K,K}$ of 
		$K$-$K$-bialgebras.	The unit object is $K[e]$, the one representing the identity functor.
	\item A \emph{composition $K$-algebra}\index{composition!algebra}\footnote{It is called a biring triple in 
		\cite{Tall-Wraith},
		a plethory in \cite{Borger-Wieland:PA}, a Tall-Wraith monad object in 
		\cite{Bergman-Hausknecht}, and a Tall-Wraith monoid in \cite{Stacey-Whitehouse:TW-monoids}.
		The term \emph{composition algebra} is both plain and descriptive; 
		so I thought to try it out here. It does however have the drawback in that it already
		exists in the literature with other meanings.} is defined to be a monoid with respect to 
		this monoidal structure. The operation is 
		denoted $\circ$ and the identity is denoted $e$.
	\item An \emph{action}\index{action!of composition algebra}
		of a composition $K$-algebra $P$ on a $K$-algebra $A$ is
		defined to be an action of the monoid object $P$, or equivalently
		of the monad $P\bcp_K\vbl$. We will write $f(a)$ for the image of
		$f\bcp a$ under the action map $P\bcp_K A\to A$. A 
		\emph{$P$-equivariant $K$-algebra}\index{$P$-equivariant algebra} 
		is a $K$-algebra equipped with an action of $P$. When $K=\bN$, we will also use the term 
		\emph{$P$-semiring}\index{semiring!$P$-semiring}.
	\item For any $K$-$L$-bialgebra $P$, we will call the functor it represents its
		\emph{Witt vector}\index{Witt!vector} functor: $W_P=\Alg{K}(P,\vbl)$.
		It takes $K$-algebras to $L$-algebras. When $K=L$,
		a composition structure on $P$ is then equivalent 
		to a comonad structure on $W_P$. When $P$ has a composition structure, then
		$W_P(A)$ has a natural action of $P$, and in this way $W_P$ can be viewed as the
		right adjoint of the forgetful functor from $P$-equivariant $K$-algebras to $K$-algebras.
\end{enumerate}

\subsection{Example: composition algebras and endomorphisms}
An element $\psi$ of a composition $K$-algebra $P$ is 
\emph{$K$-algebra-like}\index{algebra-like element} 
if for all 
$K$-algebras $A$ with an action of $P$, 
the self map $x\mapsto \psi(x)$ of $A$ is a $K$-algebra map. This is equivalent to requiring
	\begin{equation} \label{eq:algebra-like-element}
	\Delta^+(\psi)=\psi\tn 1 + 1\tn \psi, \quad\quad \Delta^\times(\psi)=\psi\tn\psi,
	\quad\quad \beta(c)(\psi)=c,
	\end{equation}
for all $c\in K$. (For comparison, one could say $d\in P$ is $K$-derivation-like if it acts as a
$K$-linear derivation on any $K$-algebra. This can also be expressed directly by saying $d$ is
primitive under $\Delta^+$ and it satisfies the Leibniz rule $\Delta^\times(d)=d\tn e+ e\tn d$ and
the $K$-linearity identity $\beta(c)(d)=0$.)

Now let $G$ be a monoid. Let $P$ be the $K$-algebra freely generated (as
an algebra) by the symbols $\psi_g$, for all $g\in G$. Then $P$ has a unique composition structure such that
each $\psi_g$ is $K$-algebra-like and we have
	$$
	\psi_g\circ \psi_h = \psi_{gh}
	$$
for all $g,h\in G$. Then an action of $P$ on an algebra $A$, in the sense of (5) of \ref{subsec:plethystic-algebra}, is the same as an action of $G$
on $A$, in the usual sense of a monoid map $G\to\Alg{K}(A,A)$. 

In this case, the Witt functor is simply $W_P(A)=A^G$, where $A^G$ has the usual product
algebra structure.

\subsection{Models of co-$\CC$ objects in $\Alg{K}$}
\label{subsec:models-of-co-objects}

Let $\CC$ be a category of the kind considered above.
Let $K\to K'$ be an $\bN$-algebra map, and let $P'$ be a co-$\CC$ object in $\Alg{K'}$. Then a
\emph{model}\index{model} for $P'$ over $K$ (or a \emph{$K$-model}\index{$K$-model}) is a co-$\CC$ object $P$ in $\Alg{K}$ 
together with an isomorphism
$K'\tn_K P \to P'$ of co-$\CC$ objects of $\Alg{K'}$. Then for any $K'$-algebra $A'$, we have
	$$
	W_{P'}(A') = \Alg{K'}(P',A') \longisomap \Alg{K}(P,A') = W_P(A').
	$$
So the Witt vector functor of $P$ extends $W_{P'}$ from $\Alg{K'}$ to $\Alg{K}$.
Conversely, any such extension to a representable functor comes from a unique model of $P'$.

\subsection{Flat models of co-$\CC$-objects in $\Alg{K}$}
\label{subsec:flat-models}
We will be especially interested in finding $K$-models of $P'$ that are flat (over $K$).
Of course, these can exist only when $P'$ is flat over $K'$, but this will be the case in all our examples. 
Further we will only consider the case where the structure map $K\to K'$ is injective.

Under these assumptions, the composition $P\to K'\tn_K P\to P'$ is injective. Conversely, if a subset
$P\subset P'$ admits a flat model structure, then it does so in a unique way. Indeed, since the
induced maps $P^{\tn_K n}\to P'^{\tn_{K'}n}$ are injective, each co-operation $\Delta$ on $P'$
restricts to at most one on $P$. One might then say that being a flat model (when
$K\subseteq K'$) is a property of a given subset of $P'$, rather than a structure on it.

The case where $\CC=\Alg{L}$ will be of particular interest to us. Then
a flat model is just a subset $P\subseteq P'$ such that the following properties hold:
\begin{enumerate}
	\item $P$ is a flat sub-$K$-algebra of $P'$,
	\item the induced map $K'\tn_{K}P \to P'$ is a bijection,
	\item $\Delta^+(P)\subseteq P\tn_{K} P$ and $\eps^+(P)\subseteq K$,
	\item $\Delta^{\times}(P)\subseteq P\tn_{K} P$ and $\eps^\times(P)\subseteq K$,
	\item $\beta(L)\subseteq \Alg{K}(P,K)$, where $\Alg{K}(P,K)$ is regarded as a subset
		of $\Alg{K'}(P',K')$.
\end{enumerate}
When $L=\bN$, the last condition is always satisfied. (Alternatively, the conditions on the
co-units $\eps^+$ and $\eps^\times$ are also redundant but not in the absence of (5).)

For other categories $\CC$, it is usually clear how to modify these conditions. For example, if
$\CC=\Mod{\bN}$, one would drop (4) and (5). 

\subsection{Models of composition algebras}
\label{subsec:models-of-composition-algebras}
When $P'$ is a composition $K'$-algebra, we will usually want to descend the composition structure
as well. Then a $K$-model of $P'$ (as a composition algebra) is a composition $K$-algebra $P$
together with an action of $P$ on $K'$ and an isomorphism $K'\tn_K P\to P'$ of composition
$K'$-algebras. Giving such a model is equivalent to extending $W_{P'}$ to a representable comonad on $\Alg{K}$.

A flat model (when $K\subseteq K'$) is then just a flat sub-$K$-algebra $P\subseteq P'$ satisfying 
conditions (1)--(5) above plus
\begin{enumerate}
	\item[(6)] $P\circ P\subseteq P$ and $e\in P$.
\end{enumerate}
So again, if $P$ admits such a structure, it does so in a unique way.

\section{The composition structure on symmetric functions over $\bN$}

The purpose of this section is to give two different $\bN$-models of $\Lambda_\bZ$, the composition
ring of symmetric functions. Since $\Lambda_\bZ$ represents the usual big Witt vector functor,
these give extensions of the big Witt vector functor to $\bN$-algebras.

Our treatment is broadly similar to Macdonald's~\cite{Macdonald:SF}. He discusses the co-additive
structure in example 25 of I.5, the co-multiplicative structure in example 20 of I.7, and plethysm
in I.8.

Let $K$ be an $\bN$-algebra.

\subsection{Conventions on partitions}
\label{subsec:partition-conventions}

We will follow those of Macdonald, p. 1~\cite{Macdonald:SF}.
So, a {\em partition}\index{partition} is an element 
	$$
	\lambda=(\lambda_1,\lambda_2,\dots)\in\bN\oplus\bN\oplus\cdots
	$$
such that $\lambda_1\geq \lambda_2\geq\cdots$. As is customary, we will allow ourselves to
omit any number of zeros, brackets, and commas and to use exponents to 
represent repetition. So for example we have $(3,2,2,1,0,\dots)=32^210=32^21$ and $(0,\dots)=0$.

The \emph{length}\index{length (of a partition)} of $\lambda$ is the smallest $i\geq 0$ such that
$\lambda_{i+1}=0$. The \emph{weight}\index{weight (of a partition)} of $\lambda$ is $\sum_i
\lambda_i$ and is denoted $|\lambda|$. We also say $\lambda$ is a partition of its weight.

\subsection{$\Psi_K$}
Let $\Psi_K$ denote the composition $K$-algebra associated to the multiplicative monoid  
of positive integers. So $\Psi_K = K[\psi_1,\psi_2,\dots]$, where each $\psi_n$ is
$K$-algebra-like and we have $\psi_m\circ \psi_n = \psi_{mn}$.

We will be interested in (flat) models of $\Psi_\bQ$ over smaller subalgebras, especially $\bQpl$,
$\bZ$ and $\bN$. There are the obvious models $\Psi_\bQpl$, $\Psi_\bZ$, and $\Psi_\bN$, but we will
be more interested in larger ones.

\subsection{Symmetric functions}

Let $\Lambda_K$ denote the usual $K$-algebra of symmetric functions in infinitely many
variables $x_1,x_2,\dots$ with coefficients in $K$. (See Macdonald, p.\ 19~\cite{Macdonald:SF}.) 
More precisely, $\Lambda_K$ is the set
of formal series $f(x_1,x_2,\dots)$ such that the terms of $f$ have bounded degree and for all $n$,
the series $f(x_1,\dots,x_n,0,0,\dots)$ is a polynomial which is invariant under permuting
the variables $x_1,\dots,x_n$. 

It is clear that $\Lambda_K$ is freely generated as a $K$-module by
the \emph{monomial}\index{monomial symmetric function}\index{symmetric function!monomial} symmetric functions $m_\lambda$, where $\lambda$ ranges over all partitions and
where
	$$
	m_\lambda = \sum_{\alpha}x_1^{\alpha_1}x_2^{\alpha_2}\cdots
	$$
where $\alpha$ runs over all permutations of $\lambda$ in $\bN^\infty$. In particular, we have 
$K\tn_{\bN}\Lambda_\bN=\Lambda_K$. When $K=\bN$, this is the unique basis of $\Lambda_\bN$, up to 
unique isomorphism on the index set.

It is well known that when $K$ is a ring, $\Lambda_K$ is freely generated as a $K$-algebra by 
the \emph{complete}\index{complete symmetric function}\index{symmetric function!complete} symmetric functions $h_1,h_2,\dots$, where
	$$
	h_n := \sum_{i_1\leq\cdots\leq i_n}x_{i_1}\cdots x_{i_n} = \sum_{|\lambda|=n}m_\lambda.
	$$
Alternatively, if we write
	$$
	\psi_n = x_1^n+x_2^n+\cdots,
	$$
then the induced map
	\begin{equation} \label{map:coghost}
		\Psi_K = K[\psi_1,\psi_2,\dots] \longmap \Lambda_K
	\end{equation}
is an injection when $K$ is a flat $\bZ$-algebra and is a bijection when $K$ is a $\bQ$-algebra. In
particular, $\Lambda_\bN$ and $\Lambda_\bZ$ are models for $\Psi_\bQ=\Lambda_\bQ$. The elements $\psi_n$ have
several names: the \emph{Adams}\index{Adams symmetric function}\index{symmetric function!Adams,
Frobenius, power sum}, \emph{Frobenius}\index{Frobenius!symmetric function}, and
\emph{power-sum}\index{power-sum symmetric function} symmetric functions.

\subsection{Remark: $\Lambda_\bN$ is not free as an $\bN$-algebra}
\label{subsec:Lambda-is-not-a-free-algebra}

Indeed, all $m_\lambda$ are indecomposable additively and, one checks, multiplicatively---except $m_0$, which
is invertible. Therefore any generating set of $\Lambda_\bN$ as an $\bN$-algebra must contain all the
$m_\lambda$ but $m_0$. But they are not algebraically independent, because any monomial in them is a linear
combination of the others. For example $m_{1}^2=m_{2}+2m_{1,1}$.

\subsection{Elementary and Witt symmetric functions}
\label{subsec:elem-and-Witt-symm-functions}
At times, we will use other families of symmetric functions, such as 
the \emph{elementary}\index{elementary symmetric function}\index{symmetric function!elementary} symmetric functions
	$$
	e_n = \sum_{i_1<\cdots<i_n}x_{i_1}\cdots x_{i_n}.
	$$
and the less well-known \emph{Witt}\index{Witt!symmetric function}\index{symmetric function!Witt} symmetric functions $\theta_1,\theta_2,\dots$, which are 
determined by the relations
	\begin{equation}
	\label{eq:Frobenius-Witt-relation}
		\psi_n = \sum_{d\mid n}d\theta_d^{n/d},
	\end{equation}
for all $n\geq 1$.
Probably the most concise way of relating them all is with generating functions
in $1+t\Lambda_\bQ[[t]]$:
	\begin{equation}
	\label{eq:generating-function-identity}
		\prod_{d\geq 1} \big(1-\theta_d t^d\big)^{-1} 
			= \exp\big(\sum_{j\geq 1}\frac{\psi_j}{j}t^j\big)
			= \sum_{i\geq 0} h_i t^i
			= \big(\sum_{i\geq 0} e_i (-t)^i\big)^{-1}.
	\end{equation}
Indeed, one can check that each expression equals $\prod_j (1-x_j t)^{-1}$. 

One notable consequence of (\ref{eq:generating-function-identity}) is that the $\theta_d$ generate $\Lambda_\bZ$ freely as a ring:
	$$
	\Lambda_{\bZ} = \bZ[\theta_1,\theta_2,\dots].
	$$
Another is that we have
	\begin{equation}
	\label{eq:complete-symms-are-psi-positive}
		\bQpl[h_1,h_2,\dots] \subseteq \Psi_\bQpl.
	\end{equation}
Note that such a containment is a special property of the complete symmetric functions $h_n$; it is not shared by the elementary symmetric functions $e_n$.

The Witt symmetric functions are rarely encountered outside the literature on Witt vectors. The
following are some decompositions in more common bases:
\begin{align*}
	\theta_1 &= \psi_1 = m_1 = e_1 = h_1 = s_1 \\
	\theta_2 &= (-\psi_1^2+\psi_2)/2 = -m_{1^2} = -e_2 = -h_1^2+h_2 = -s_{1^2} \\
	\theta_3 &= (-\psi_1^3+\psi_3)/3 = -2m_{1^3}-m_{21} = -e_2e_1+e_3 = -h_2h_1+h_3 = -s_{21} \\
	\theta_4 &= (-3\psi_1^4 + 2\psi_2\psi_1^2 - \psi_2^2 + 2\psi_4)/8 
		= -9m_{1^4} - 4m_{21^2} - 2m_{2^2} - m_{31}\\
	 	&= -e_2e_1^2 + e_3e_1 - e_4 = -h_1^4 + 2h_2h_1^2 - h_2^2 - h_3h_1 + h_4 \\
		&= -s_{1^4} - s_{21^2} - s_{2^2} - s_{31} \\
	\theta_5 &= (-\psi_1^5+\psi_5)/5 
		= -24m_{1^5} - 12m_{21^3} - 6m_{2^21} - 4m_{31^2} - 2m_{32} - m_{41} \\
		&= -e_2e_1^3 + e_2^2e_1 + e_3e_1^2 - e_3e_2 - e_4e_1 + e_5 \\
		&= -h_2h_1^3 + h_2^2h_1 + h_3h_1^2 - h_3h_2 - h_4h_1 + h_5 \\
		&= -s_{21^3} - s_{2^21} - s_{31^2} - s_{32} - s_{41} \\
	\theta_6
	 	&= (-13\psi_1^6 - 9\psi_2\psi_1^4 + 9\psi_2^2\psi_1^2 - 3\psi_2^3 + 8\psi_3\psi_1^3
		-4\psi_3^2 + 12\psi_6)/72 \\
		&=	-130m_{1^6} -68m_{21^4}-35m_{2^21^2}-18m_{2^3} - 24m_{31^3} - 12m_{321} - 4m_{3^2} \\
		& \quad - 6m_{41^2} - 3m_{42} - m_{51} \\
		&= -e_2e_1^4 + e_2^2e_1^2 + e_3e_1^3 - e_3e_2e_1 - e_4e_1^2 + e_4e_2 + e_5e_1 - e_6 \\
		&= -h_2^2h_1^2 - h_3h_1^3 + 3h_3h_2h_1 - h_3^2 + h_4h_1^2 - h_4h_2 - h_5h_1 + h_6 \\
		&= -s_{2^21^2} - s_{2^3} - 2s_{31^3} - 3s_{321} - s_{3^2} - 2s_{41^2} - 2s_{42} - s_{51}
\end{align*}
Observe that the coefficients in some bases are noticeably smaller than in others. It would be interesting to make this precise.

\begin{proposition}\label{pro:lambda-structure-maps}
	Consider the composition algebra structure on $\Lambda_\bQ$ induced by
	the isomorphism $\Psi_\bQ\to\Lambda_\bQ$ of (\ref{map:coghost}).
	Then the structure maps on $\Lambda_\bQ$ satisfy the following:
	\begin{align*}
		\Delta^+\: f(\dots,x_i,\dots)&\mapsto f(\dots,x_i\tn 1,1\tn x_i,\dots) \\
		\Delta^{\times}\: f(\dots,x_i,\dots)&\mapsto f(\dots,x_i\tn x_j,\dots) \\
		\eps^+\: f(\dots,x_i,\dots) &\mapsto f(0,0,\dots) \\
		\eps^\times\: f(\dots,x_i,\dots) &\mapsto f(1,0,0,\dots) \\
		f\circ g &= f(y_1,y_2,\dots),
	\end{align*}
	for all $f\in\Lambda_\bQ$ and $g\in\Lambda_\bN$, where
	$g=y_1+y_2+\cdots$ with each $y_j$ a monomial in the variables $x_i$ with coefficient $1$. 	
\end{proposition}
\begin{proof}
	The maps $\Delta^+$, $\Delta^\times$, $\eps^+$, $\eps^\times$ displayed above
	are manifestly $\bQ$-algebra maps. Therefore to show they agree with the corresponding
	structure maps,
	it is enough to consider elements $f$ that run over a set of generators, such as the $\psi_n$.
	In this case, one can calculate the image of $\psi_n$ under the maps displayed above and 
	observe that it agrees with the image of $\psi_n$ under the structure maps, by 
	(\ref{eq:algebra-like-element}).
	
	Similarly, for any fixed $g\in\Lambda_\bN$, the map $f\mapsto f\circ g$ displayed above is
	a $\bQ$-algebra map. So it is again enough to assume $f=\psi_m$. 
	Then $\psi_m\circ g\mapsto g$ is an $\bN$-algebra map $\Lambda_\bN\to\Lambda_\bQ$,
	and  to show this map
	agrees with the rule displayed above, it is enough to do so after
	allowing coefficients in $\bQ$, since $\Lambda_\bN$ is free. Then the $\psi_n$ are generators,
	and so it is enough to assume $g=\psi_n$.
	We are left to check $\psi_m\circ\psi_n=\psi_{mn}$, which is indeed
	the composition law on $\Psi_\bQ$.
\end{proof}

\begin{corollary}\label{cor:models-of-lambda-Q}
	$\Lambda_\bN$ and $\Lambda_\bZ$ are models of the composition algebra 
	$\Lambda_\bQ$ in a unique way.
\end{corollary}
\begin{proof}
	Since both $\Lambda_\bN$ and $\Lambda_\bZ$ are free, it is enough to check properties (1)--(6)
	of~(\ref{subsec:flat-models}) and~(\ref{subsec:models-of-composition-algebras}).
	It is immediate from (\ref{pro:lambda-structure-maps}) that all the structure maps preserve
	$\Lambda_\bN$, which finishes the proof for it. 
	
	It is also clear that all
	the structure maps preserve $\Lambda_\bZ$, with the exception of the one expressing that
	$\Lambda_\bZ$ is closed
	under composition. We will now show this.
	Because $\Lambda_\bN$ is a composition algebra and we have 
	$\Lambda_\bZ = \bZ\tn_\bN\Lambda_\bN$, it is enough to show $f(\bZ)\subseteq \bZ$ for all 
	$f\in\Lambda_\bZ$, and hence $f(-1)\in\bZ$ for all such $f$.
	To do this, it is enough to restrict to the case where $f$ ranges over
	a set of algebra generators, such as the $h_n$. But since $\psi_n(-1)=-1$, the 
	identity~(\ref{eq:generating-function-identity}) implies $h_n(-1)=0$ for $n\geq 2$ and 
	$h_1(-1)=-1$. So we have $h_n(-1)\in\bZ$ for all $n$.
\end{proof}

\subsection{Explicit description of a $\Lambda_\bN$-action}
\label{subsec:explicit-lambda_N-actions}

For the convenience of the reader, let us spell out what it means for the composition algebra
$\Lambda_\bN$ to act on an $\bN$-algebra $A$ in elementary terms.\index{semiring!$\Lambda_\bN$-semiring}\index{action!$\Lambda_\bN$-action}

For each partition $\lambda$, there is a set map
	$$
	m_\lambda\:A \longmap A
	$$
such that the following identities hold
	\begin{align*}
	m_0(x) &= 1 \\
	(m_\lambda m_\mu)(x) &= m_\lambda(x) m_\mu(x) \\
	m_\lambda(0) &= \eps^+(m_\lambda) \\
	m_\lambda(x+y) &= \Delta^+(m_\lambda)(x,y) 
		:= \sum_{\mu,\nu} q^\lambda_{\mu\nu} m_\mu(x)m_\nu(y)\\
	m_\lambda(1) &= \eps^\times(m_\lambda) \\
	m_\lambda(xy) &= \Delta^\times(m_\lambda)(x,y)
	 	:= \sum_{\mu,\nu} r^\lambda_{\mu\nu} m_\mu(x)m_\nu(y) \\
	m_1(x) &= x \\
	m_\lambda(m_\mu(x)) &= (m_\lambda\circ m_\mu)(x).
	\end{align*}
The notation here requires some explanation. The expression $(m_\lambda m_\mu)(x)$, or more generally $f(x)$, is defined by
	$$
	f(x) = \sum_\nu f_\nu m_\nu(x) \in A,
	$$
where the numbers $f_\nu\in\bN$ are defined by the equation $f=\sum_\nu f_\nu m_\nu$;
and $q^\lambda_{\mu\nu}$ and $r^\lambda_{\mu\nu}$ are the structure coefficients in $\bN$
for the two coproducts:
	\begin{align*}
		\Delta^+(m_\lambda) &= \sum_{\mu,\nu} q^\lambda_{\mu\nu} m_\mu\tn m_\nu\\
		\Delta^\times(m_\lambda) &= \sum_{\mu,\nu} r^\lambda_{\mu\nu} m_\mu\tn m_\nu. 
	\end{align*}

\subsection{Example: the toric $\Lambda_\bN$-structure on monoid algebras}
Let $A$ be a commutative monoid, written multiplicatively.
The monoid algebra $\bN[A]$ is the set of finite formal sums $\sum_{i=1}^n[a_i]$,
where each $a_i\in A$, and multiplication is the linear map satisfying the law $[a][b]=[ab]$.
Then for any $f\in\Lambda_\bN$ define
	\begin{equation}
	\label{eq:lambda-action-on-monoid-algebra}
	f(\sum_{i=1}^n [a_i]) := f([a_1],\dots,[a_n],0,0,\dots).
	\end{equation}
The right-hand side denotes the substitution $x_i=[a_i]$ into the symmetric function $f$.
To show this law defines an action of $\Lambda_\bN$ on $\bN[A]$, it is enough to observe
that~(\ref{eq:lambda-action-on-monoid-algebra}) is the restriction of an action of
a larger composition algebra on a larger $\bN$-algebra. By~(\ref{pro:lambda-structure-maps}), 
it is the restriction of the action of $\Lambda_\bQ$ on 
$\bQ[A]:=\bQ\tn_\bN \bN[A]$ determined by $\psi_n([a])=[a^n]$ for all $a\in A$, $n\geq 1$.

We call the $\Lambda_\bN$-structure on $\bN[A]$ the \emph{toric}\index{toric!
$\Lambda_\bN$-structure} $\Lambda_\bN$-structure because $\Spec(\bZ[A])$ is a toric variety and the
$\Lambda_\bN$-structure extends in a canonical way to any nonaffine toric variety, once this concept
is defined. (See~\cite{Borger:LRFOE}, for example.)

\subsection{Example: the Chebyshev line}

When $A=\bZ$ above, we have $\bN[A]=\bN[x^{\pm 1}]$. Let $B$ denote the subalgebra of $\bN[A]$
spanned by $1,x+x^{-1},x^2+x^{-2},\dots$. Then $B$ is the set of invariants under the involution
$\psi_{-1}$ of $\bN[x^{\pm 1}]$ defined by $x\mapsto x^{-1}$. Since $\psi_{-1}$ is a
$\Lambda_\bN$-morphism and since the category of $\Lambda_\bN$-algebras has all limits, $B$ is a
sub-$\Lambda_\bN$-semiring. It is a model over $\bN$ of the $\Lambda_\bZ$-ring $\bZ[x+x^{-1}]$
called the {\em Chebyshev line}\index{Chebyshev line} in~\cite{Borger:LRFOE}, but it is not isomorphic to the $\bN$-algebra
$\bN[y]$. In fact, it is not even finitely generated as an $\bN$-algebra.

\subsection{Flatness for $\Lambda_\bN$-semirings}
In~\cite{Borger-deSmit:Integral-models}, De Smit
and I proved some classification results about reduced $\Lambda$-rings that are finite flat over
$\bZ$. It would be interesting to know if there are similar results over $\bN$.
As D.\ Grinberg pointed out to me, it is not hard to construct
non-toric $\Lambda_\bN$-semirings which are flat and
finitely presented over $\bN$. One example is the sub-$\bN$-algebra
of $\psi_{-1}$-invariants of the monoid algebra $\bN[x]/(x^3=1)$. It is isomorphic to
$\bN[y]/(y^2=y+2)$, via the map $y\mapsto x+x^2$.

\subsection{Example: convergent exponential monoid algebras}

Let $B$ be a submonoid of $\bC$ under addition, and let $\Nhat{e^B}$ denote the set of formal series
$\sum_{i=1}^\infty [e^{b_i}],$ where $b_i\in B$, such that the complex series $\sum_i e^{b_i}$
converges absolutely. (Of course, $e$ now denotes the base of the natural logarithm.) We identify
series that are the same up to a permutation of the terms. More formally, $\Nhat{e^B}$ is the set
of elements $(\dots,n_b,\dots)\in\bN^B$ such that the sums $\sum_{b\in S}n_b |e^b|$ are bounded as
$S$ ranges over all finite subsets of $B$. (Also, note that the notation is slightly abusive in
that $\Nhat{e^B}$ depends on $B$ itself and not just its image under the exponential map.) It is a
sub-$\bN$-module of $\bN^B$ and has a multiplication defined by
	$$
	\Big(\sum_{i} [e^{b_i}]\Big)\Big(\sum_{j} [e^{c_j}]\Big)
	= \Big(\sum_{i,j} [e^{b_i+c_j}]\Big),
	$$
under which it becomes an $\bN$-algebra. It is the unique multiplication extending that
on the monoid algebra $\bN[B]$ which is continuous, in some suitable sense.

For integers $n\geq 1$, define
	$$
	\psi_n\Big(\sum_i[e^{b_i}]\Big) = \sum_i[e^{nb_i}].
	$$
This is easily seen to be an element of $\Nhat{e^B}$, and hence the $\psi_n$ form a commuting
family of $\bN$-algebra endomorphisms of $\Nhat{e^B}$.
The induced endomorphisms of $\bQ\tn_\bN\Nhat{e^B}$ prolong to a unique action of $\Lambda_\bQ$.
It follows that for any symmetric function $f\in\Lambda_\bQ$, we have
	\begin{equation}
	\label{eq:lambda-action-on-convergents}
	f\big(\sum_i [a_i]\big) = f([a_1],[a_2],\dots) \in \bQ\tn_\bN\bN^B.
	\end{equation}
Indeed, it is true when $f=\psi_n$; then by the multiplication law above, it
is true when $f$ is a monomial in the $\psi_n$, and hence for any $f\in\Psi_\bQ=\Lambda_\bQ$.
It follows that for any $f\in\Lambda_\bN$, the element $f(\sum_i[a_i])$ lies in $\bN^B\cap
(\bQ\tn_\bN\Nhat{e^B})=\Nhat{e^B}.$ So $\Nhat{e^B}$ inherits an action of $\Lambda_\bN$ from
$\bQ\tn_\bN\Nhat{e^B}$.

Observe that if $B$ is closed under multiplication by any real $t\geq 1$, then we have operators $\psi_t$, for all $t\geq 1$:
	$$
	\psi_t\big(\sum_i [e^{b_i}]\big) = \sum_i [e^{tb_i}].
	$$
Indeed the series $\sum_i e^{tb_i}$ is absolutely convergent if $\sum_i e^{b_i}$ is.
Thus $\Spec(A)$ has a flow interpolating the Frobenius operators. For further remarks, see~(\ref{subsec:positivity-philosophy}).

\subsection{Example: convergent monoid algebras}
\label{subsec:convergent-monoid-alg}
Now let $A$ be a submonoid of the positive real numbers under multiplication, and let
$\log A=\setof{b\in\bR}{e^b\in A}$. Then define
	$$
	\Nhat{A} = \Nhat{e^{\log A}}.
	$$

\subsection{Remark: non-models for $\Lambda_\bZ$ over $\bN$}
\label{subsec:non-models}
The $\bN$-algebra $\bN[e_1,\dots]$ is a model for $\Lambda_\bZ$ as a co-$\bN$-module object
because the expression $$\Delta^+(e_n)=\sum_{i+j=n} e_i\tn e_j$$ has no negative coefficients. But it is not a 
model as a co-$\bN$-algebra object because the analogous expression for $\Delta^\times$ does at places have negative coefficients:
	$$
	\Delta^\times(e_2) = e_2\tn e_1^2 + e_1^2\tn e_2 - 2e_2\tn e_2.
	$$

The complete symmetric functions
	$$
	h_n = \sum_{i_1\leq\cdots\leq i_n}x_{i_1}\cdots x_{i_n}
	$$
behave similarly. The $\bN$-algebra $\bN[h_1,\dots]$ is a model for $\Lambda_\bZ$ over $\bN$ 
as a co-$\bN$-module object, because $\Delta^+(h_n)=\sum_{i+j=n}h_i\tn h_j$, but not
as a co-$\bN$-algebra object, because 
$\Delta^{\times}(h_2)=h_1^2\tn h_1^2-h_1^2\tn h_2 - h_2\tn h_1^2 + 2h_2\tn h_2$.

For Witt symmetric functions, we have
	\begin{align*}
		\Delta^+(\theta_p) &= \theta_p\tn 1 + 1\tn\theta_p - 
			\sum_{i=1}^{p-1}\frac{1}{p}\binom{p}{i}\theta_1^i \tn \theta_1^{p-i} \\
		\Delta^\times(\theta_p) &= \theta_p\tn\theta_1^p + \theta_1^p\tn\theta_p + 
			p\theta_p\tn\theta_p,	
	\end{align*}
for any prime $p$. So $\bN[\theta_1,\dots]$ is not even a model as a 
co-$\bN$-module object. Using $-\theta_p$ instead of $\theta_p$, removes the sign from 
the first formula but adds one in the second. We will see in section~\ref{sec:p-typical-model}
that it is possible to circumvent this problem, at least if we care only about a single prime.

\section{The Schur model for $\Lambda_\bZ$ over $\bN$}

\subsection{Schur functions and $\Sch$}
For any partition $\lambda=(\lambda_1,\dots,\lambda_r)$, put
	\begin{equation}
		s_\lambda := \det(h_{\lambda_i-i+j}) \in \Lambda_\bZ,
	\end{equation}
where as usual $h_0=1$ and $h_n=0$ for $n<0$. For example, 
	\begin{equation}
		s_{321}=\det\left(
		\begin{array}{lll}
			h_3 & h_4 & h_5 \\
			h_1 & h_2 & h_3 \\
			0 & 1 & h_1
		\end{array}
		\right).
	\end{equation}
Such symmetric functions are called \emph{Schur functions}\index{Schur function}. They have simple
interpretations in terms of representation theory of general linear and symmetric groups, and
several of the results we use below are usually proved using such an interpretation. But it is
enough for us just to cite the results, and so we will ignore this interpretation. This is
discussed more in~(\ref{subsec:representation-theory}) below.

Write
	\begin{equation}
		\Sch := \bigoplus_\lambda \bN s_\lambda,
	\end{equation}
where $\lambda$ runs over all partitions. The Schur functions are well known to form a $\bZ$-linear
basis for $\Lambda_\bZ$. (See p.\ 41 (3.3) of Macdonald's book \cite{Macdonald:SF}.) Therefore
$\Sch$ is an $\bN$-model for $\Lambda_\bZ$ as a module, by way of the evident inclusion
$\Sch\to\Lambda_\bZ$.

\begin{proposition}
\label{pro:Schur-is-a-model}
$\Sch$ is an $\bN$-model	of $\Lambda_\bZ$ as a composition algebra, in a unique way. 
We also have $\Sch\subseteq\Lambda_\bN$.
\end{proposition}
\begin{proof}
	Since $\Sch$ is free, it is enough to check properties (1)--(6)
	of~(\ref{subsec:flat-models}) and~(\ref{subsec:models-of-composition-algebras}).
	These all reduce to standard facts about Schur functions, for which
	we will refer to chapter 1 of Macdonald's book \cite{Macdonald:SF}.

	We have $\Sch\subseteq \Lambda_\bN$ because of the equality
		\begin{equation}
		\label{eq:Kostka}
			s_\lambda = \sum_\mu K_{\lambda\mu}m_\mu
		\end{equation}
	where the $K_{\lambda\mu}$ are the Kostka numbers, which are in $\bN$, by 
	Macdonald, p.\ 101 (6.4).

	The fact that $\Sch$ is a sub-$\bN$-algebra follows from the equations
	$$
	1=s_0, \quad\quad s_{\lambda}s_{\mu} = \sum_\nu c^{\nu}_{\lambda\mu}s_\nu,
	$$
	where the $c^{\nu}_{\lambda\mu}$ are the Littlewood--Richardson coefficients, which
	are in $\bN$, by Macdonald, pp.\ 113--114, (7.5) and (7.3).

	By Macdonald, p.\ 72 (5.9) and p.\ 119 (7.9), we have
		\begin{align}
			\label{eq:LW-coeffs}
			\Delta^+(s_\lambda) &= \sum_{\mu\nu}c^{\lambda}_{\mu\nu}s_\mu \tn s_\nu \\
			\label{eq:Kronecker-coeffs}
			\Delta^\times(s_\lambda) &= \sum_{\mu\nu}\gamma^{\lambda}_{\mu\nu}s_\mu \tn s_\nu,
		\end{align}
	where the $c^{\nu}_{\lambda\mu}$ are again the Littlewood--Richardson coefficients and
	the $\gamma^{\lambda}_{\mu\nu}$ are the Kronecker coefficients, which
	are in $\bN$ by pp.\ 114--115 of Macdonald. It follows that $\Sch$ is closed under the 
	co-addition	and comultiplication maps.
	
	The containments $\eps^+(\Sch),\eps^{\times}(\Sch)\subseteq\bN$ follow from 
	$\Sch\subseteq\Lambda_\bN$.

	Finally, we have
		$$
		s_{\lambda}\circ s_{\mu} = \sum_\nu a^{\nu}_{\lambda\mu}s_\nu,
		$$
	where $a^{\nu}_{\lambda\mu}\in\bN$, by Macdonald, p.\ 136 (8.10). 
	So we have $s_\lambda\circ s_{\mu}\in\Sch$.
	It follows that $\Sch$ is closed under composition, since $\Sch$ is an $\bN$-$\bN$-bialgebra.
	Last, it contains the compositional identity element $e$ because $e=s_1$.
\end{proof}

\subsection{Remark} 

The three models of $\Psi_\bQ$ over $\bN$ we have seen are $\Psi_\bN$, $\Lambda_\bN$, and $\Sch$.
The largest of the three is $\Lambda_\bN$. The other two are incomparable, even over $\bQpl$.
Indeed, we have $s_{11}=(\psi_1^2-\psi_2)/2$ and $\psi_2=s_1^2-2s_{11}$.

\subsection{Remark: $\Sch$ is not free as an $\bN$-algebra}
\label{subsec:Sch-is-not-a-free-algebra}

As in~(\ref{subsec:Lambda-is-not-a-free-algebra}), any generating set of $\Sch$ contains all Schur
functions that cannot be written as monomials with coefficient $1$ in the other Schur functions.
This is true of $s_1=h_1$, $s_2=h_2$, and $s_{11}=h_1^2-h_2$ because they are irreducible in
$\Lambda_\bZ=\bZ[h_1,h_2,\dots]$. But they are not algebraically independent: $s_1^2=s_2+s_{11}$.

\subsection{Remark: explicit description of a $\Sch$-action}
\label{subsec:explicit-Sch-action}
This is the same as in~(\ref{subsec:explicit-lambda_N-actions}), but with operators $s_\lambda$.
\index{semiring!$\Sch$-semiring}\index{action!$\Sch$-action}
Now the identities are the following:
	\begin{align*}
	s_0(x) &= 1 \\
	(s_\lambda s_\mu)(x) &= s_\lambda(x) s_\mu(x) \\
	s_\lambda(0) &= \eps^+(s_\lambda) \\
	s_\lambda(x+y) &= \Delta^+(s_\lambda)(x,y) 
		:= \sum_{\mu,\nu} c^\lambda_{\mu\nu} s_\mu(x)s_\nu(y)\\
	s_\lambda(1) &= \eps^\times(s_\lambda) \\
	s_\lambda(xy) &= \Delta^\times(s_\lambda)(x,y)
	 	:= \sum_{\mu,\nu} \gamma^\lambda_{\mu\nu} s_\mu(x)s_\nu(y) \\
	s_1(x) &= x \\
	s_\lambda(s_\mu(x)) &= (s_\lambda\circ s_\mu)(x) 
		:= \sum_{\nu} a^\nu_{\lambda\mu} s_\mu(x)s_\nu(y),
	\end{align*}
where $c^{\lambda}_{\mu\nu}$,  $\gamma^\lambda_{\mu\nu}$, and $a^\lambda_{\mu\nu}$ are the 
Littlewood--Richardson, Kronecker, and Schur plethysm coefficients, as in the proof 
of~(\ref{pro:Schur-is-a-model}).

In these terms, a $\Lambda_\bN$-action gives rise to a $\Sch$-action by the 
formula~(\ref{eq:Kostka}).

\subsection{Example} 

Since we have $\Sch\subseteq\Lambda_\bN$, any $\Lambda_\bN$-semiring 
is a $\Sch$-semiring. For example, in a monoid $\bN$-algebra $\bN[M]$ with the toric 
$\Lambda_\bN$-structure, we have
	\begin{equation}
	\label{eq:schur-of-line-element}
	s_\lambda([m]) = \left\{ 
		\begin{array}{ll} 
			[m]^l & \text{if }\lambda=(l) \\ 0 & \text{otherwise} 
		\end{array} 
		\right.
	\end{equation}

\subsection{Example} 
\label{subsec:anti-teich-schur-semiring}
Consider the monoid $\bZ$-algebra $\bZ[x]/(x^2-1)$ with the $\Lambda_\bZ$-structure
defined by $\psi_n(x)=x^n$, and let $A$ denote the sub-$\bN$-algebra generated
by $\eta=-x$. So we have a presentation 
	$$
	A=\bN[\eta]/(\eta^2=1).
	$$
Then $A$ is a sub-$\Sch$-semiring. Indeed, for any partition $\lambda$ of $r$, we have
	$$
	s_\lambda(-x) = (-1)^r s_{\lambda'}(x) =  \left\{ 
		\begin{array}{ll} (-x)^r & \text{if }\lambda=(1^r) \\ 0 & \text{otherwise} \end{array} 
		\right.
	$$
by (\ref{eq:schur-of-line-element}), where $\lambda'$ is the partition conjugate to $\lambda$.

\subsection{Frobenius lifts and $p$-derivations}
For each integer $n\geq 1$, we have $\psi_n = m_n \in \Lambda_\bN.$ Therefore every
$\Lambda_\bN$-semiring $A$ has a natural endomorphism $\psi_n$. It is called the {\em Adams}\index{Adams operator} or {\em Frobenius}\index{Frobenius!operator}
{\em operator} . When $n$ is a prime $p$, the induced map on $\bF_p\tn_\bN A$ is the Frobenius map
$x\mapsto x^p$. To see this, one can reduce to the case of rings, or one can simply observe that we
have
	$$
	d_p := \frac{1}{p}(\psi_1^p - \psi_p) \in \Lambda_\bN
	$$
and hence $\psi_p(x)+pd_p(x)=x^p$.
The operator $-d_p=\theta_p$ is a $p$-derivation in the language of Buium, p.\ 
31~\cite{Buium:Arithmetic-diff-equ}. 

It is different for $\Sch$-semirings. Since $\psi_n\not\in\Sch$ for $n\geq 2$, we do not
generally have Frobenius lifts. However we do have the operators $d_n:=-\theta_n$ for $n\geq 2$.
Indeed Scharf--Thibon~\cite{Scharf-Thibon} and Doran~\cite{Doran:Reutenauer} proved a conjecture of
Reutenauer~\cite{Reutenauer:Witt} that
	$$
	-\theta_n\in\Sch \quad \text{for}\quad n\geq 2.
	$$

Of course, for $n=1$ we have the opposite: $+\theta_1\in\Sch$. This irksome exception suggests the
perspective here might not be the most fundamental one. For example, it is a standard result in the
theory of $p$-typical Witt vectors (tantamount to the Cartier--Dieudonn\'e--Dwork lemma) that the
ring $\bZ[d_{p^n}| n\geq 0]$ of $p$-typical symmetric functions agrees with $\bZ[d_p^{\circ
n}\mid n\geq 0]$, the composition $\bZ$-algebra generated by the iterated $p$-derivations. The
families of generators $d_{p^n}$ and $d_p^{\circ n}$ agree for $n\leq 1$ but not for $n=2$.

\begin{question}
Do the symmetric functions $d_{p^n}-d_p^{\circ n}$ lie in $\Sch$? More generally, do
$d_{p^{m+n}}-d_{p^m}\circ d_{p^n}$? If not, do they lie in $\Lambda_\bN$?
\end{question}

If the answers are yes, it would suggest that the operators $d_p^{\circ n}$ are more fundamental
from the point of view of positivity than the $d_{p^n}$. In other words, up to some signs, the
iterated $p$-derivations generate a larger sub-$\bN$-algebra than the Witt symmetric functions.
Computations have shown that the answer to the stronger question above is yes when $p=2$, $m+n\leq
3$; $p=3$, $m+n=2$; and $p=5$, $m+n=2$.

The obvious generalization $d_{rs}-d_r\circ d_s\in\Sch$ is not always true. Computations
have shown it is true for $(r,s)=(2,3),(2,5),(3,5),(5,3)$ but false for $(3,2)$ and $(5,2)$.

\subsection{Remark: the necessity of nonlinear operators}

Observe that neither $\Lambda_\bN$ nor $\Sch$ is generated by linear operators, even after base
change to $\bRpl$. Indeed, because we have $m_{11}=e_2=s_{11}$, any generating set of
$\Lambda_\bRpl$ or $\bRpl\tn_\bN\Sch$ would have to contain a nonzero multiple of $e_2$, which is
nonadditive element. So an action of $\Lambda_\bZ$ or $\Sch$ on an $\bN$-algebra cannot be expressed
entirely in terms of additive operators.

It appears to be the case that over $\bZ$ or $\bZ_p$, any composition algebra that cannot be
generated by linear operators can trace its origin to the $d_p$ operators---in other words, to
lifting Frobenius maps to characteristic $0$. There are probably theorems to this effect. (For
example, Buium~\cite{Buium:arithmetic-analogues} classifies ring scheme structures on the plane
$\bA^2_{\bZ_p}$, which is perhaps the first test case, and he is able to prove such a result
there.) Could it be that the existence of nonlinear composition algebras over $\bRpl$ is due to a
similarly identifiable phenomenon? It seems optimistic to hope that the answer is yes, but if it
were, the importance would be so great that the question should not be dismissed.

\subsection{Remark: composition algebras over number fields}
\label{subsec:lambda-over-number-fields}

The composition ring $\Lambda_\bZ$ has analogues over rings of integers in general number fields.
We will not really use them in this chapter, but they will appear in a several remarks and open
questions.

Let $K$ be a number field, let $\sO_K$ denote its ring of integers, and let $E$ denote a family of
maximal ideals of $\sO_K$. Let $\Lambda_{\sO_K,E}$ denote the composition $\sO_K$-algebra
characterized by the property that an action on any flat $\sO_K$-algebra $A$ is the same as a
commuting family of $\sO_K$-algebra endomorphisms $(\psi_\gp)_{\gp\in E}$ such that each $\psi_\gp$
reduces to the Frobenius map $x\mapsto x^{[\sO_K:\gp]}$ on $A/\gp A$. This construction and the
associated Witt vector functor $W_{\sO_K,E}$ are discussed in much more detail in section 1 of my
paper~\cite{Borger:BGWV-I}. When $K=\bQ$, it reduces to a special case of the construction
of~(\ref{subsec:truncation-sets}), as is explained there, but otherwise there is no overlap between
the two.

\subsection{Remark: representation theory and K-theory}
\label{subsec:representation-theory}

The standard way of looking at $\Sch$ is from the point of view of polynomial functors, or equivalently
representations of $\GL_n$. (See Macdonald~\cite{Macdonald:SF}, ch.\ 1, app.\ A.) Let $\mathcal{F}$ denote the
category of polynomial functors from the category of finite-dimensional vector spaces over, say, $\bC$ to
itself. Then for any polynomial functor $F$, the polynomial $\tr(F(\diag(x_1,\dots,x_n)))$ in
$\bZ[x_1,\dots,x_n]$ is symmetric, and as $n$ varies, these polynomials define a compatible sequence and hence
an element of the inverse limit $\Lambda_\bZ$. This defines a group homomorphism $\chi\:K(\mathcal{F})\to
\Lambda_\bZ$, where $K(\mathcal{F})$ denotes the Grothendieck group of $\mathcal{F}$, and one proves this map
is bijection.

Under this bijection the irreducible polynomial functors correspond to the Schur functions. Indeed, given a
partition $\lambda$ of $n$, let $V_\lambda$ be the corresponding irreducible $\bQ$-linear representation of
$\mathbf{S}_n$. Then the functor $S_\lambda(E):=\Hom_{\bQ[\mathbf{S}_n]}(V_\lambda,E^{\tn n})$ is a polynomial
functor. Further, the $S_\lambda$ are precisely the irreducible polynomial functors, and each $S_\lambda$
corresponds to the Schur function $s_\lambda$ under $\chi$. Therefore $\chi$ induces an isomorphism
$K_+(\mathcal{F})\to\Sch$ of $\bN$-modules, where $K_+(\vbl)$ denotes the subset of the Grothendieck
group consisting of effective classes, rather than just virtual ones. Then each of the structure
maps in the composition-algebra structure on $\Sch$ corresponds to something transparent on $\mathcal{F}$: the
operations $+$, $\times$, and plethysm on $\Sch$ correspond to $\oplus$, $\otimes$, and composition on
$\mathcal{F}$. The co-operation $\Delta^+$ corresponds to the rule that sends a polynomial functor $F$ to
the polynomial functor $F(E\oplus E')$ in two variables $E$ and $E'$; similarly $\Delta^+$ corresponds to
$F(E\otimes E')$. The positivity properties of $\Sch$ used in the proof of (\ref{pro:Schur-is-a-model}) below
then hold because the (co-)operations on $\Sch$ correspond to (co-)operations on $\mathcal{F}$, and hence they
preserve effectivity.

This point of view is also a good way of approaching the $\lambda$-ring structure on Grothendieck groups. Given
an amenable linear tensor category $\CC$, we can define functors $S_\lambda\:\CC\to\CC$ by
$S_\lambda(E)=\Hom_{\bQ[\mathbf{S}_n]}(V_\lambda,E^{\tn n})$. This defines an action of the monoid
$K_+(\mathcal{F})=\Sch$ under composition on $K_+(\CC)$, and this is essentially by definition an action of the
composition $\bN$-algebra $\Sch$ on the $\bN$-algebra $K_+(\CC)$. This, in turn, induces an action of
$\Lambda_\bZ$ on $K(\CC)$, which is the usual $\lambda$-ring structure on Grothendieck groups.

While this is a more conceptual way of thinking about positivity properties on $\Sch$, the deeper meaning of
the connection with Witt vectors still eludes me. From an arithmetic point of view, it is not clear why one
would should consider the Witt vector functor associated to the composition algebra made by assembling all
this representation-theoretic data into the algebraic gadget $\Lambda_\bZ$. Indeed, the composition algebras
$\Lambda_{\sO_K,E}$ of~(\ref{subsec:lambda-over-number-fields}) and the corresponding Witt functors
$W_{\sO_K,E}$ also have arithmetic interest (for instance in the theory of complex multiplication when $K$ is
an imaginary quadratic field), but no representation theoretic interpretation of $\Lambda_{\sO_K,E}$ is known.
In fact, no interpretation in terms of something similar to symmetric functions is known.

At the time of this writing, there is still a tension between the following three observations: (1)
the connection between Frobenius lifts and the representation theory of $\GL_n$ appears to be a
coincidence, merely an instance of a fundamental algebraic object arising in two unrelated
contexts; (2) to define $\Sch$ and $\Wsch$ one needs the positivity results established with
representation theory; (3) $\Sch$ and $\Wsch$ have some arithmetic interest, as in
section~\ref{sec:total-positivity} below. The situation with $\Lambda_\bN$ and $W$ is perhaps less
mysterious---they seem more important from the arithmetic point of view and also require no
nontrivial positivity results from representation theory. Perhaps the resolution will be that
$\Sch$ and $\Wsch$ are of arithmetic interest only through their relation to $\Lambda_\bN$ and $W$.

Whatever the case, this is why I have ignored K-theory and used the representation theory as a
black box. However this point of view does suggest the following question:

\begin{question}
	\label{ques:p-poly-func}
Let $B_p$ denote the subset of $\Lambda_\bZ$ consisting of the characters of polynomial functors over an 
infinite field of characteristic $p>0$. Is $B_p$ a model for $\Lambda_\bZ$ as a composition algebra?
\end{question}

The structure of $B_p$ is apparently much subtler than that of its analogue $\Sch$ in characteristic $0$.
Nevertheless I expect the answer to this question to be \emph{yes} for formal reasons, as with $\Sch$ above.
It would be interesting to make a more detailed study of the $B_p$ from the point of view of plethystic algebra.

\section{Witt vectors of $\bN$-algebras}

For the classical theory of Witt vectors for rings, one can see Bergman's lecture 26 in Mumford's
book~\cite{Mumford:Lectures-on-curves}, or chapter III of Hazewinkel's book~\cite{Hazewinkel:book},
or \S1 of chapter IX of Bourbaki~\cite{Bourbaki:CommAlg} and especially the exercises there. One
can also see Witt's original writings,~\cite{Witt:Vectors} and pp.\ 157--163 of
\cite{Witt:CollectedPapers}.

\subsection{$W$ and $\Wsch$}
For any $\bN$-algebra $A$, define the {\em $\bN$-algebra of Witt vectors}\index{$\bN$-algebra!of Witt vectors} with entries in $A$ by
\index{Witt!vector}\index{Schur--Witt vector}
	$$\index{$W(A)$}
	W(A) := \Alg{\bN}(\Lambda_\bN,A),
	$$
and define the 
{\em $\bN$-algebra of Schur--Witt vectors} with entries in $A$ by
	$$
	\Wsch(A) := \Alg{\bN}(\Sch,A).
	$$
The $\bN$-algebra structures are inherited from the $\bN$-$\bN$-bialgebra structure on $\Lambda_\bN$ and 
$\Sch$, as explained in section~\ref{sec:plethystic-algebra}.
Since both $\Lambda_\bN$ and $\Sch$ are models of $\Lambda_\bZ$ over $\bN$, both Witt functors agree up to canonical isomorphism with the usual big Witt vector functor on rings. 

Since $\Lambda_\bN$ is a composition algebra, $W(A)$ has a natural $\Lambda_\bN$-action. 
It is defined, for $a\in W(A)$ and $f\in\Lambda_\bN$, by
	\begin{equation}
		f(a)\:g\mapsto a(g\circ f)
	\end{equation}
for all $g\in\Lambda_\bN$. Equivalently, the comonad structure map $W(A)\to W(W(A))$ is given by
$a\mapsto[g\mapsto a(g\circ f)]$. The analogous statements hold for $\Wsch$ and $\Sch$.

Finally, the two functors $W$ and $\Wsch$ are related: the inclusion 
$\Sch\subseteq\Lambda_\bN$ of composition algebras induces a map 
	$$
	W(A)\longmap \Wsch(A)
	$$
of $\Sch$-semirings which is natural in $A$.

\subsection{The ghost map and similar ones}
Consider the diagram
	$$
	\xymatrix{
	\Sch \ar@{>->}[r] & \Lambda_\bN \\
	\bN[h_1,h_2,\dots] \ar@{>->}[u]\ar@{>-->}[r] & \Psi_\bN \ar@{>->}[u]
	}
	$$
of sub-$\bN$-algebras of $\Lambda_\bZ$, where the bottom map is dashed to indicate that
it is defined only after base change to $\bQpl$.
If we apply the functor $\Alg{\bN}(\vbl,A)$ to this diagram, we get
the following diagram of Witt vectors:
	\begin{equation} 
		\label{diag:ghost-and-series-maps}
		\begin{split}
	\xymatrix{
	\Wsch(A) \ar_-{\ser}[d] & W(A) \ar_-{\schurmap}[l]\ar^-{w}[d] \\
	1+tA[[t]] &  A^\infty \ar@{-->}_-{\sergh}[l],
	}
	\end{split}
	\end{equation} 
where in the bottom row we have made use of the following identifications
	\begin{align*}
		\Alg{\bN}(\bN[h_1,\dots],A)&\longisomap 1+tA[[t]], \quad 
			a\mapsto \sum_i a(h_i)t^i \\
		\Alg{\bN}(\Psi_\bN,A) &\longisomap A^\infty, \quad 
			a\mapsto \langle{a(\psi_1),a(\psi_2),\dots\rangle}.
	\end{align*}
The dashed arrow indicates that $A$ needs to be a $\bQpl$-algebra for $\sergh$ to be defined.
Recall from~(\ref{subsec:non-models}) that $1+tA[[t]]$ is in general only a commutative 
monoid, under usual power-series multiplication; but the other three are $\bN$-algebras. 
\begin{proposition}
\begin{enumerate}
	\item $\schurmap$ is a morphism of $\Sch$-semirings, and $w$ is a morphism of
	 	$\Psi_\bN$-semirings, and 
		$\ser$ and $\sergh$ are morphisms of
		$\bN$-modules: $\sigma_*(x+y)=\sigma_*(x)\sigma_*(y)$, $\sigma_*(0)=1$.
	\item If $A$ is a $\bZ$-algebra, then $\schurmap$ and $\ser$ are bijections.
		If $A$ is contained in a $\bZ$-algebra, then $\schurmap$ and $\ser$ are injections.
	\item If $A$ is a $\bQ$-algebra, then $w$ and $\sergh$ are bijections.
		If $A$ is contained in a $\bQ$-algebra, then $w$ and $\sergh$ are injections.
\end{enumerate}
\end{proposition}
\begin{proof}
	(1): The first statement holds because the inclusions $\Sch\to\Lambda_\bN$ and 
		$\Psi_\bN\to\Lambda_\bN$ are morphisms of composition algebras.
		The second follows from the equations $\Delta^+(h_n)=\sum_{i+j=n}h_i\tn h_j$ and
		$\eps^+(h_n)=0$ for $n\geq 1$.

	(2): This holds because $\Lambda_\bN$, $\Sch$, and $\bN[h_1,\dots]$
		are models for $\Lambda_\bZ$ over $\bN$.

	(3): Similarly, this holds because $\Lambda_\bN$, $\Psi_\bN$, and $\bN[h_1,\dots]$ 
		are models for $\Psi_\bQ$ over $\bN$.
\end{proof}

\subsection{Remark} 

The map $\schurmap\:W(A)\to\Wsch(A)$ can fail to be injective when $A$ is not contained in a ring.
For example, it is not injective when $A$ is the Boolean semiring $\bN/(1+1=1)$. See the
forthcoming work with Darij Grinberg~\cite{Borger-Grinberg:W(Bool)}. 

\subsection{Coordinates for Witt vectors of rings}
\label{subsec:W-coords-for-rings}
When $A$ is a ring, the bijection $\ser$ of (\ref{diag:ghost-and-series-maps}) allows us to 
identify Witt vectors and power series. In
other words, the complete symmetric functions give a free set of coordinates on the
$\bZ$-scheme $W=\Spec(\Lambda_\bZ)$.

Note that there are three other common conventions for identifying Witt vectors and power series. 
Given any two signs $\eps_1,\eps_2\in\{+1,-1\}$, the map
	\begin{align}
	\label{map:series-normalizations}
	\begin{split}
	\ser_{\eps_1}^{\eps_2}\:W(A) &\longisomap 1+tA[[t]] \\
	a &\mapsto \big(\sum_i a(e_i)(\eps_1t)^i\big)^{\eps_2}
	\end{split}
	\end{align}
is an isomorphism of $\bN$-modules. We have taken $\ser=\ser^-_-$ as our convention. 
It has one advantage over the others, which is that $\sergh$ 
can be defined over $\bQpl$ rather than just $\bQ$.
Equivalently, the complete symmetric functions can
be written as polynomials in the power sums with nonnegative coefficients. This is not
true with the elementary symmetric functions, even if we allow ourselves to change the signs
of half of them. I am not aware of any other convincing reason to prefer one of these sign conventions to the 
others.

The presentation $\Lambda_\bZ=\bZ[\theta_1,\dots]$ gives another set of full coordinates:
	$$
	W(A) \longisomap A^\infty, \quad a \mapsto (a(\theta_1), a(\theta_2),\dots)
	$$
We will call the elements $a(\theta_1),\dots\in A$ the 
\emph{Witt components}\index{Witt!components, coordinates} or \emph{Witt coordinates}
of $a$. Addition does not have a simple closed form in these coordinates, unlike in
the series coordinates above. But the Witt coordinates do have two advantages over the series
coordinates. First, they are related
to the power sums by a simple, closed-form expression~(\ref{eq:Frobenius-Witt-relation}).
Second, they behave well with respect to localization. More precisely, let $E$ be a set
of prime numbers. Then we have a free presentation
	$$
	\Lambda_{\bZ[1/E]} = \bZ[1/E][\dots,\theta_{m}\circ\psi_n,\dots],
	$$
where $m$ runs over the positive integers whose prime divisors do not lie in $E$, and
$n$ runs over those whose prime divisors all do. (See the isomorphisms
(1.20.1) and (5.3.1) in~\cite{Borger:BGWV-I}, for example.)

The components of the image of $a$ under the map 
	$$
	w\:W(A)\to A^\infty, \quad a\mapsto \langle{a(\psi_1),a(\psi_2),\dots}\rangle
	$$
are its \emph{ghost components}\index{ghost components}. Their advantage is that all algebraic structure is completely
transparent---addition and multiplication are performed componentwise and Frobenius operations are
given by scaling the indices. When $A$ is a $\bQ$-algebra, $w$ is a bijection, and so the ghost
components have no deficiencies over $\bQ$. When $A$ is a torsion-free ring (i.e.\ a flat
$\bZ$-algebra), $w$ is an injection, and the ghost coordinates are still useful there. We write the
ghost components with angle brackets to avoid any confusion with the Witt components.

\subsection{Witt vectors for semirings contained in rings}
A Witt vector is determined by its series coordinates when $A$ is contained in a  
ring. Indeed, a map $\Lambda_\bN\to A$ is equivalent to a map $a\:\Lambda_\bZ\to\bZ\tn_\bN A$
subject to the \emph{effectivity condition}\index{effectivity condition} $a(m_\lambda)\in A$ for all partitions $\lambda$.
Therefore we have
	\begin{equation}
	\label{eq:effectivity-conditions-for-W(A)}
	W(A) = \bigsetof{\sum_i a_i t^i \in 1+tA[[t]]}{\sum_\mu M^\mu_\lambda a_{\mu_1}a_{\mu_2}\cdots 
		\in A},
	\end{equation}
where $M$ is the transition matrix from the monomial basis $(m_\lambda)_\lambda$ to the 
basis of monomials in the complete symmetric functions $h_n$:
	$$
	m_\lambda = \sum_\mu M^\mu_\lambda h_{\mu_1}h_{\mu_2}\cdots.
	$$
Formula~(\ref{eq:effectivity-conditions-for-W(A)}) also holds for $\Wsch(A)$, but then $M$ must be the transition matrix 
from the Schur basis to this $h$ basis.

Similar statements hold for the ghost components when $A$ is contained in a
$\bQ$-algebra. For example, we have
	\begin{equation}
	\Wsch(A) = \bigsetof{\langle{a_1,a_2,\dots}\rangle\in A^\infty}{\sum_\mu N^\mu_\lambda 
		a_{\mu_1}a_{\mu_2}\cdots \in A},
	\end{equation}
where $N$ is the transition matrix from the Schur basis to the basis of monomials
in the power sums:
	$$
	s_\lambda = \sum_\mu N^\mu_\lambda \psi_{\mu_1}\psi_{\mu_2}\cdots.
	$$
In this particular case, the matrix $N$ has a well-known description: 
the block where $|\lambda|=|\mu|=n$ is the inverse of the
character table of the symmetric group $\mathbf{S}_n$. See Macdonald, (7.8) p.~114~\cite{Macdonald:SF}.

\subsection{Example: some explicit effectivity conditions}
Let us write out the effectivity conditions on ghost components for $\Wsch(A)$ and
$W(A)$ up to weight $4$. If a ghost vector $\langle{a_1,\dots,a_4}\rangle\in A^4$ lies in
the image of $\Wsch(A)$, then the following eleven elements of
$\bQ\tn_\bN A$ are contained in $A$:
	$$
	a_1 , \quad
	(a_1^2\pm a_2)/2 , \quad
	(a_1^3 \pm 3a_1a_2 + 2a_3)/6  , \quad
	(a_1^3-a_3)/3, 
	$$
	$$
	(a_1^4 \pm 6a_1^2a_2 + 3a_2^2 + 8a_1a_3 \pm 6a_4)/24  ,
	$$
	$$
	(3a_1^4 \pm 6a_1^2a_2-3a_2^2\mp 6a_4)/24  , \quad
	(2a_1^4+6a_2^2-8a_1a_3)/24.
	$$
These are all the conditions corresponding to $s_\lambda$ with $|\lambda|\leq 4$.

In the case of $W(A)$, the following elements are required to be in $A$:
$$
	a_1, \quad
	a_2, \quad
	a_3, \quad
	a_4, \quad
$$
$$	
	(a_1^2-a_2)/2,  \quad
	a_1a_2 - a_3, 
$$
$$
	(a_1^3 - 3a_1a_2 + 2a_3)/6 ,  \quad
	a_1a_3 - a_4,  \quad
	(a_2^2 - a_4)/2, \quad
$$
$$
	(a_1^2a_2 - a_2^2 - 2a_1a_3 + 2a_4)/2, \quad
	(a_1^4 - 6a_1^2a_2 + 3a_2^2 +8a_1a_3 - 6a_4)/24
$$
These are all the conditions corresponding to $m_\lambda$ with $|\lambda|\leq 4$.

\subsection{Coordinates for Witt vectors of semirings}
For general semirings, there is no purely vector-like description of $W(A)$ or $\Wsch(A)$. Indeed,
for any set $T$, the functor $A\mapsto A^T$ is represented by the free $\bN$-algebra on $T$, but
by~(\ref{subsec:Lambda-is-not-a-free-algebra}) and~(\ref{subsec:Sch-is-not-a-free-algebra}), neither
$\Lambda_\bN$ nor $\Sch$ is free as an $\bN$-algebra.

Instead Witt vectors are cut out of infinite-dimensional affine space by quadratic relations
determined by the structure constants of multiplication in the relevant basis.
Let $\prt$ denote the set of partitions. For any $\lambda,\mu\in P$, write
	$$
	m_\lambda m_\mu = \sum_{\nu\in P} b^\nu_{\lambda\mu} m_\nu, \quad\quad
	s_\lambda s_\mu = \sum_{\nu\in P} c^\nu_{\lambda\mu} s_\nu.
	$$
Then we have
	\begin{align*}
		W(A) &= \setof{a\in A^{\prt}}{a_0=1,\ a_\lambda a_\mu = \sum_\nu b^\nu_{\lambda\mu}a_\nu} \\
		\Wsch(A) &= \setof{a\in A^{\prt}}{a_0=1,\ a_\lambda a_\mu = \sum_\nu c^\nu_{\lambda\mu}a_\nu}.
	\end{align*}
Addition and multiplication are then defined using the structure constants for
the coproducts $\Delta^+$ and $\Delta^\times$ with respect to the basis in question.

\subsection{Topology and pro-structure} 
\label{subsec:W-topology} 
It is often better to view $W(A)$ and $\Wsch(A)$ as pro-sets, or pro-discrete topological spaces, 
as when $A$ is a ring. We do this as follows. Let $(P_i)_{i\in I}$
denote the filtered system of finitely generated sub-$\bN$-algebras of $\Lambda_\bN$. Then we have
\begin{equation} 
	\label{eq:W-pro-structure} 
	W(A)=\Hom(\colim_{i\in I}P_i,A)=\lim_{i\in I}\Hom(P_i,I). 
\end{equation} 
So $W(A)$, and similarly $\Wsch(A)$, has the natural structure of a pro-set. 
When $A$ is a ring, $W(A)$ can be expressed as an inverse limit of rings. 
I do not know, however, whether the analogous statement holds over $\bN$:

\begin{question}
Are $W(A)$ and $\Wsch(A)$ pro-objects in the category of $\bN$-algebras? More naturally, is this
true as representable functors? Equivalently, can $\Lambda_\bN$ and $\Sch$ be expressed as filtered
colimits $\colim_i P_i$, where each $P_i$ is a co-$\bN$-algebra object in $\Alg{\bN}$ which is
finitely generated as an $\bN$-algebra?
\end{question}

For a stronger form of this question, see~(\ref{subsec:truncation-sets}).

\subsection{Teichm\"uller and anti-Teichm\"uller elements}
\label{subsec:Teich}

For any $\bN$-algebra $A$, consider the monoid algebra $\bN[A]$ with the toric 
$\Lambda_\bN$-structure. By the adjunction property of $W$, the $\bN$-algebra map $\bN[A]\to A$ 
defined by $[a]\mapsto a$ lifts to a unique $\Lambda_\bN$-equivariant map
	$$
	\bN[A]\longmap W(A).
	$$
For $a\in A$, the image of $[a]$ in $W(A)$ is called the \emph{Teichm\"uller lift}\index{Teichm\"uller lift} of $a$ and is
also denoted $[a]$. Explicitly, $[a]$ is the $\bN$-algebra map $\Lambda_\bN\to A$ determined by
	\begin{equation}
	[a]:m_\lambda \longmap 
		\begin{cases}
			a^r & \text{if $\lambda=(r)$} \\
			0 & \text{otherwise}
		\end{cases}
	\end{equation}
The Teichm\"uller map $a\mapsto [a]$ is a map of monoids $A\to W(A)$ under
multiplication.

For the anti-Teichm\"uller elements, consider the $\bN$-algebra 
	$$
	\bN[\eta]/(\eta^2=1) \tn_\bN \bN[A];
	$$ 
this has an action of $\Sch$. On the first factor, $\Sch$ acts as
in~(\ref{subsec:anti-teich-schur-semiring}). On the second factor, it acts through the toric action
of $\Lambda_\bN$. Then the $\bN$-algebra map
	$$
	\bN[\eta]/(\eta^2=1) \tn_\bN \bN[A] \longmap A
	$$
determined by $\eta\mapsto 1$ and $[a]\mapsto a$ for all $a\in A$ lifts by adjunction to a unique
$\Sch$-equivariant map
	$$
	\bN[\eta]/(\eta^2=1) \tn_\bN \bN[A] \longmap \Wsch(A).
	$$
For any $a\in A$, define the \emph{anti-Teichm\"uller lift}\index{anti-Teichm\"uller lift} $\anti{a}\in\Wsch(A)$ 
to be the image of $\eta\tn a$. So in $W(\bZ\tn_\bN A)$ we have
	\begin{equation}
	\anti{a}=-[-a].
	\end{equation}
We also have
	\begin{equation}
	\anti{a}:s_\lambda \longmap 
		\begin{cases}
			a^r & \text{if $\lambda=(1^r)$} \\
			0 & \text{otherwise}
		\end{cases}
	\end{equation}
and
	\begin{equation}
	[a]\anti{b}=\anti{ab}, \quad\quad \anti{a}\anti{b}=[ab], \quad\quad \anti{a}=\anti{1}[a].
	\end{equation}
Observe that the anti-Teichm\"uller lifts exist only in $\Wsch(A)$ and not generally in $W(A)$. For 
example, the element $\anti{1}\in\Wsch(\bN)$ is not in the sub-$\bN$-algebra $W(\bN)$. Indeed,
its ghost vector $\langle{1,-1,1,-1,\dots}\rangle$ is not even in $\bN^\infty$.

\subsection{The involution and the forgotten symmetric functions}
Let
	$$
	\omega\:\Lambda_\bZ\longmap \Lambda_\bZ
	$$
denote the ring map determined by $\omega(h_n)=e_n$ for all $n$. Then we have
	$$
	\omega(s_\lambda)=s_{\lambda'}
	$$
for all $\lambda$, where $\lambda'$ denotes the conjugate partition. 
(See Macdonald, p.\ 23 (2.9)$'$~\cite{Macdonald:SF}.) Therefore we have
	$$
	\omega(\Sch) = \Sch,
	$$
and so $\omega$ induces a functorial set map
	$$
	\Wsch(A) \longmap \Wsch(A), \quad a \mapsto a\circ \omega.
	$$
In fact, this map is simply multiplication by $\anti{1}$. Indeed, it follows from the identities 
of~(\ref{eq:generating-function-identity}) that
	$$
	\omega(\psi_n) = (-1)^{n+1}\psi_n
	$$
for all $n$; now combine this with the equality $\anti{1}=\langle{1,-1,1,-1,\dots}\rangle$.

The symmetric functions $f_\lambda:=\omega(m_\lambda)$ are sometimes called the \emph{forgotten}\index{forgotten symmetric function}\index{symmetric function!forgotten} 
symmetric functions. Their span $\Forg:=\omega(\Lambda_\bN)$ contains $\Sch$. 
It represents the functor
	$$
	\Alg{\bN}(\Forg,A) = \anti{1}W(A),
	$$
which is the free $W(A)$-module generated by the symbol $\anti{1}$. The induced map
	$$
	\anti{1} W(A) \longmap \Wsch(A)
	$$
is a $W(A)$-module map, but $\anti{1}W(A)$ cannot be given an $\bN$-algebra structure making the map
an $\bN$-algebra map. In particular,
$\Forg$ is not a model for $\Lambda_\bZ$ as co-$\bN$-algebra object.

\subsection{Example: the map $\bN\to \Wsch(A)$ is injective unless $A=0$}
\label{subsec:N-to-W(A)-is-injective}
In fact, the map $\varphi$ from $\bN$ to any nonzero $\Sch$-semiring is injective. 
For if $m,n\in\bN$ have the same image and $m<n$, then we have
	$$
	0 = \varphi(\binom{m}{n}) = s_{1^n}(\varphi(m)) = s_{1^n}(\varphi(n)) = \varphi(\binom{n}{n}) = 1. 
	$$
So $\varphi$ is injective unless it is the map to the zero ring. In particular, the map from $\bN$ to any nonzero $\Lambda_\bN$-semiring is injective.

On the other hand, the $\Sch$-equivariant map $\bN[\eta]/(\eta^2=1)\to\Wsch(A)$ 
sending $\eta\to\anti{1}$ is not always injective.
Indeed, when $A=\bZ/2\bZ$, we have $[-1]=[1]=1$ and hence $1+\eta\mapsto 0$.

\section{Total positivity}
\label{sec:total-positivity}

In this section, we will give explicit descriptions of $W(\bRpl)$ and $\Wsch(\bRpl)$ and then use 
this to describe $W(\bN)$. These are very rich objects, and there is much more to say about
them than we can here.

We will find it convenient to use the series normalization
$\ser^+_+$ of~(\ref{subsec:W-coords-for-rings}), as well as
our standard one $\ser=\ser^-_-$. By~(\ref{eq:generating-function-identity}), we have
	$$
	\ser(x)  = \sum_i x(h_i)t^i, \quad\quad	\ser^+_+(x) = \sum_i x(e_i)t^i.
	$$
The two are related by the involution $f(t)\mapsto f(-t)^{-1}$. In other words, we have
	\begin{equation}
	\label{eq:change-series-normalization}
		\ser^+_+(x) = \ser(\anti{1}x).
	\end{equation}

\begin{proposition}\label{pro:mon-pos-coeff-bound}
	For any Witt vector $x\in W(\bRpl)$, write $1+a_1 t+ a_2t^2+\cdots$ 
	for the series $\ser^+_+(x)\in 1+t\bRpl[[t]]$. Then for all $n$, we have 
		\begin{equation}
		\label{eq:mon-pos-coeff-bound}
		a_n \leq \frac{a_1^n}{n!}.
		\end{equation}
	In particular, the series $\sum_n a_n t^n$ converges to an entire function on $\bC$.
\end{proposition}

\begin{proof}
	Since $m_{2,1^{n-2}} = e_{n-1}e_1-n e_n$, we have 
		$
		a_{n-1}a_1-n a_n \geq 0.
		$
	Then (\ref{eq:mon-pos-coeff-bound}) follows by induction. 
\end{proof}

\subsection{Total positivity}
A formal series $\sum_n a_n t^n\in 1+t\bR[[t]]$ is said to be 
\emph{totally positive}\index{totally positive (power series)} if
all (finite) minors of the infinite matrix $(a_{i-j})_{ij}$ are $\geq 0$. (To be clear, we 
understand $a_0=1$ and $a_n=0$ for $n<0$. Other authors allow more general series.) 
For example, up to the $2\times 2$ minors these inequalities amount to the following:
	$$
	a_n \geq 0, \quad\quad a_n a_{n+i-j} \geq a_{n+i} a_{n-j}.
	$$
Chapter 8 of Karlin's book~\cite{Karlin:book} has an extensive treatment.

\begin{proposition}
\label{pro:W-Schur-is-totally-pos}
	A Witt vector $x\in W(\bR)$ lies in $\Wsch(\bRpl)$ if and only if the series corresponding
	to $x$ under the bijection $\ser\:W(\bR)\to 1+t\bR[[t]]$ is totally positive.
	The same is true for the bijection $\ser^+_+$.
\end{proposition}
\begin{proof}
	First consider the universal series $1+h_1 t+ h_2 t^2+\cdots\in\Lambda_\bZ[[t]]$.
	Then it is a standard fact in algebraic combinatorics that the
	minors of the matrix $(h_{i-j})_{ij}$ generate $\Sch$ as an $\bN$-module.
	(The minors are the so-called {\em skew Schur functions}\index{skew Schur functions} $s_{\lambda/\mu}$, by
	Macdonald, p.\ 70 (5.4)~\cite{Macdonald:SF}. Their $\bN$-span contains $\Sch$ because
	every Schur function is a skew Schur function.
	For the other containment, see (9.1) and the following text on p.\ 142 of Macdonald.)
	Therefore $x$, viewed as a ring map $\Lambda_\bZ\to\bR$, sends $\Sch$ to $\bRpl$ if
	and only if the corresponding series is totally positive.
	
	For the second statement, combine the above with~(\ref{eq:change-series-normalization}) and
	the fact that $\anti{1}$ is an element of $\Wsch(A)$ and is invertible.
\end{proof}

\begin{theorem}[Edrei~\cite{Edrei:totally-positive}, Thoma~\cite{Thoma:totally-positive}]
\label{thm:Edrei-Thoma}
	The totally positive series in $1+t\bR[[t]]$ are precisely those of the form
		\begin{equation}
		\label{eq:Edrei-Thoma}
		e^{\gamma t} \frac{\prod_{i=1}^\infty(1+\alpha_i t)}{\prod_{i=1}^{\infty} (1-\beta_i t)},
		\end{equation}
	where $\gamma,\alpha_i,\beta_i\geq 0$ (and both $\sum_i \alpha_i$ and $\sum_i\beta_i$ 
	converge), and every such representation is unique.
\end{theorem}

For a proof, see Karlin's book~\cite{Karlin:book}, in which the result is theorem 5.3, p.\ 412.

\begin{corollary}
\label{cor:Witt-Edrei-Thoma}
	Consider a Witt vector $x\in W(\bR)$. Then $x$ lies in $\Wsch(\bRpl)$ if and only if
	$\ser^+_+(x)$ is of the form
		\begin{equation}
		e^{\gamma t} \frac{\prod_{i=1}^\infty(1+\alpha_i t)}{\prod_{i=1}^{\infty} (1-\beta_i t)},
		\end{equation}
	where $\gamma,\alpha_i,\beta_i\geq 0$. Similarly, $x$ lies in $W(\bRpl)$ if and only if 
	$\ser^+_+(x)$ is of the form
		\begin{equation}
		\label{eq:monomial-Edrei-Thoma}
		e^{\gamma t} \prod_{i=1}^\infty(1+\alpha_i t),
		\end{equation}
	where $\gamma,\alpha_i\geq 0$.
\end{corollary}
\begin{proof}
	The first part follows from~(\ref{pro:W-Schur-is-totally-pos}) and~(\ref{thm:Edrei-Thoma}).

	Now consider the second part.
	The first part and~(\ref{pro:mon-pos-coeff-bound}) imply that for every Witt vector
	$x\in W(\bRpl)$, the series
	$\ser^+_+(x)$ is of the form~(\ref{eq:monomial-Edrei-Thoma}).
	Conversely, because we have $1+\alpha t = \ser^+_+([\alpha])$ and
		$$
		e^{\gamma t} = \lim_{n\to\infty} \Big(1+\frac{\gamma t}{n}\Big)^n,
		$$
	any series of the form~(\ref{eq:monomial-Edrei-Thoma}) is a limit of finite products
	of series in $\ser^+_+(W(\bRpl))$. Now observe that $\ser^+_+$ identifies $W(\bRpl)$
	with a submonoid of $1+t\bRpl[[t]]$. Further it is a closed subset
	because it is defined by a family of nonnegativity conditions, as 
	in~(\ref{eq:effectivity-conditions-for-W(A)}). Therefore any series of the 
	form~(\ref{eq:monomial-Edrei-Thoma}) lies in the image of $\ser^+_+$.
	(Compare Kingman~\cite{Kingman:partition-structures,Kingman:coalescent}.)
\end{proof}

\subsection{Remark: $W(\bRpl)$ and $\Wsch(\bRpl)$ as convergent monoid algebras} 
\label{rmk:ring-expression-of-E-T}
Equivalently, the subset $\Wsch(\bRpl)$ consists of the Witt vectors in $W(\bR)$ that
can be represented (necessarily uniquely) in the form
	$$
	\sum_{i=1}^\infty[\alpha_i]+\sum_{i=1}^\infty
	\anti{\beta_i}+[\gamma]\xi,
	$$
where $\gamma,\alpha_i,\beta_i\geq 0$ (and both $\sum_i \alpha_i$ and $\sum_i\beta_i$ 
converge) and where $\xi$ is the Witt vector with ghost
components $\langle{1,0,0,\dots}\rangle$, or equivalently such that $\ser(\xi)=e^t$.
Another interpretation is that the evident map is an isomorphism
	$$
	\Nhat{\bRpl} \tn_\bN (\bN[\eta]/(\eta^2=1)) \oplus \bRpl\xi\ \longisomap\ \Wsch(\bRpl).
	$$
Similarly, we have $W(\bRpl)=\Nhat{\bRpl}\oplus\bRpl\xi$.

\subsection{Remark: $W(\bRpl)$ and entire functions}
If we view (\ref{eq:monomial-Edrei-Thoma}) as a Hadamard factorization, then we see that yet
another interpretation of $W(\bRpl)$ is that it is the set of entire functions $f$ on $\bC$ of
order at most $1$ such that $f(0)=1$, the zeros of $f$ are negative real numbers, and $p=0$
in the notation of lecture 4 of Levin's book~\cite{Levin:Entire-functions}.

\begin{corollary}
\label{cor:W-of-discrete-is-polys}
	$W(\bN)$, viewed as a subset of 
	$1+t\bZ[[t]]$ by the map $\ser^+_+$,
	agrees with the set of polynomials in $1+t\bZ[t]$ whose complex roots are all real and negative.
	In particular, $W(\bN)$ is countable.
\end{corollary}
\begin{proof}
	We have
		$
		W(\bN) = W(\bZ\cap\bRpl) = W(\bZ) \cap W(\bRpl).
		$
	By~(\ref{cor:Witt-Edrei-Thoma}), elements on the right-hand side correspond to series of the 
	form~(\ref{eq:monomial-Edrei-Thoma}) with coefficients in $\bZ$. Certainly this includes
	all the polynomials in $1+t\bZ[t]$ with only negative real roots. Conversely, 
	the coefficients of such a series tend to $0$, by (\ref{pro:mon-pos-coeff-bound});
	so all such series are polynomials.
\end{proof}

\subsection{Remark} 

We can reinterpret this in a way that treats the finite and infinite places of $\bQ$ as similarly
as possible. A monic $p$-adic polynomial has coefficients in $\bZ_p$ if and only if all its roots
are integral over $\bZ_p$. Therefore $W(\bN)$, viewed as a subset of $1+t\bQ[[t]]$ via $\ser^+_+$,
is the set of polynomials that when written as $\prod_i(1+\alpha_i t)$, have the property that
every $\alpha_i$ is integral at each finite place and is real and positive at the infinite place.

Rephrasing again, if $\sO_{\bar{\bQ}}^{\mathrm{tp}}$ denotes the multiplicative monoid of algebraic
numbers which are integral at all finite places and which are real and positive at all infinite
places, then we have
	\begin{equation}
	\label{eq:W(N)-formula}	
		W(\bN) = \bN[\sO_{\bar{\bQ}}^{\mathrm{tp}}]^{\Gal(\bar\bQ/\bQ)}.
	\end{equation}

\subsection{Counterexample: $W$ does not preserve surjectivity}
Indeed, $\bN$ surjects onto a nonzero ring, for instance $\bZ/2\bZ$. But $W(\bN)$ is countable
while $W$ applied to any nonzero ring is uncountable. Of course, $W$ does preserve surjectivity
for maps between rings, by~(\ref{subsec:W-coords-for-rings}).

In fact, $\Wsch(\bN)$ is also countable. This will be shown in forthcoming work
with Darij Grinberg~\cite{Borger-Grinberg:W(Bool)}. It follows that $\Wsch$ does not preserve surjectivity 
either. 

\subsection{Remark}
There is a multi-dimensional generalization of~(\ref{cor:W-of-discrete-is-polys}). Let
$A$ be a discrete subring of $\bR^n$, for some $n\geq 0$, and write $A_+=A\cap\bRplex{n}$. Then
$W(A_+)$, viewed as a subset of $W(A)=1+tA[[t]]$, consists of the polynomials which split 
completely over $\bRplex{n}$. In particular, this applies to any totally real number field $K$,
in which case we have the following generalization of~(\ref{eq:W(N)-formula}): 
	\begin{equation}
	\label{eq:W(tot-real)-formula}	
		W(\sO_{K,+}) = \bN[\sO_{\bar{\bQ}}^{\mathrm{tp}}]^{\Gal(\bar{\bQ}/K)}.
	\end{equation}
In particular, there are countably many Witt vectors with entries in the algebra of
algebraic integers that are real and nonnegative at all infinite places.

\section{A model for the $p$-typical symmetric functions over $\bN$}
\label{sec:p-typical-model}

\subsection{$p$-typical Witt vectors and general truncation sets}
\label{subsec:truncation-sets}

Following Bergman (lecture 26 of~\cite{Mumford:Lectures-on-curves}),
let us say a set $S$ of positive integers is a \emph{truncation set}\index{truncation set} if it is closed 
under taking divisors. For any truncation set $S$, write 
	$$
	\Lambda_{\bZ,S}:=\bZ[\theta_d\mid d\in S] \subseteq \bZ[\theta_1,\theta_2,\dots] =\Lambda_\bZ,
	$$
where the $\theta_d$ are the Witt symmetric functions of~(\ref{subsec:elem-and-Witt-symm-functions}). 
For any ring $A$, write
	$$
	W_S(A):=\Alg{\bZ}(\Lambda_{\bZ,S},A)
	$$
for the corresponding ring of Witt vectors.  

The induced map $W(A)\to W_S(A)$ is surjective for all $A$. Indeed, any retraction
$\Lambda_\bZ\to\Lambda_{\bZ,S}$ gives a functorial section. The quotient $W_S(A)$ of $W(A)$ is in fact a
quotient ring, or equivalently $\Lambda_{\bZ,S}$ is a sub-$\bZ$-$\bZ$-algebra of $\Lambda_\bZ$. One can show
this as follows: By induction on $S$, the ring $\bQ\tn_\bZ\Lambda_{\bZ,S}$ agrees with $\bQ[\psi_d\mid d\in
S]$, which is a sub-$\bQ$-$\bZ$-algebra of $\Psi_{\bQ}$; therefore $\Delta^+(\Lambda_{\bZ,S})$ is contained in
$\Lambda_\bZ^{\tn 2} \cap (\bQ\tn_\bZ\Lambda_{\bZ,S})^{\tn 2}$ and hence $\Lambda_{\bZ,S}^{\tn 2}$.

This functor $W_S\:\Alg{\bZ}\to\Alg{\bZ}$ then agrees with the usual one in
Bergman~\cite{Mumford:Lectures-on-curves}, at least up to canonical isomorphism. Similarly, it agrees with the
Witt vector functors of~\cite{Borger:BGWV-I}, as long as $S$ is of the form $\setof{d}{d\text{ divides }m}$ for
some integer $m\geq 1$. (The functors in~\cite{Borger:BGWV-I} are defined only in that context. Note however
that such truncation sets form a cofinal family.) More precisely, $W_S$ and $\Lambda_{\bZ,S}$ agree with the
objects $W_{\bZ,E,n}$ and $\Lambda_{\bZ,E,n}$ defined in~\cite{Borger:BGWV-I}, section 1, where $E$ denotes the
set of prime divisors of $m$, and $n\in\bN^E$ is the vector with components $n_p=\ord_p(m)$.

When $S=\{1,p,\dots,p^k\}$ for some prime $p$, the functor $W_S$ is the $p$-typical Witt vector
functor of length $k$ (or more traditionally $k+1$) discussed in the introduction. In this case, we
will write $W_{(p),k}(A)$ and $\Lambda_{\bZ,(p),k}$ instead of $W_S$ and $\Lambda_{\bZ,S}$. When
$k=\infty$, we will also write $W_{(p)}$ and $\Lambda_{\bZ,(p)}$.

\begin{question}
Does $\Lambda_{\bZ,S}$ have a model $\Lambda_{\bN,S}$ over $\bN$ as a co-$\bN$-algebra object? If
so, are there models such that $\Lambda_{\bN,S}\circ\Lambda_{\bN,S'}\subseteq \Lambda_{\bN,SS'}$,
where $SS'=\setof{ss'}{s\in S, s'\in S'}$? 
\end{question}

The purpose of this section is to show the answers are {\em yes} in the $p$-typical case, when $S$ and
$S'$ contain only powers of a single prime $p$.

\subsection{Positive $p$-typical symmetric functions}
\label{subsec:p-typical-symmetric-functions}
Let $p$ be a prime number. Recall the notation
	$$
	d_p:=\frac{e^p-\psi_p}{p} = m_{p-1,1} + \cdots + (p-1)! m_{1^p} \in \Lambda_\bN\cap 
		\Lambda_{\bZ,(p),1}.
	$$
Let $\lttn{k}$ denote the sub-$\bN$-algebra of $\Lambda_{\bZ,(p),k}$
generated by the set $\setof{\psi_p^{\circ i}\circ d_p^{\circ j}}{i+j\leq k}$, and write
	$$
	A_k = \bN[x_{i,j}\mid 0\leq i+j\leq k]/(x_{i,j}^p=x_{i+1,j}+px_{i,j+1} \mid i+j\leq k-1).
	$$
So $A_k$ is an algebra over $\bN[x_{i,j}\mid i+j= k]$; as a module, it is free of rank 
$p^{\frac{k(k+1)}{2}}$.

\begin{lemma}
	\label{lem:p-typical-model}
	The $\bN$-algebra map $\bN[x_{i,j}\mid i+j\leq k]\to \lttn{k}$ sending $x_{i,j}$ to 
	$\psi_p^{\circ i}\circ d_p^{\circ j}$ factors through $A_k$, and the induced map
	is an isomorphism
		\begin{equation}
		\label{map:lambda_p,k-presentation}
		A_k \longisomap \lttn{k}.
		\end{equation}
	In particular, $\bN[\psi_p^{\circ i}\circ d_p^{\circ j}\mid i+j=k]$ is freely generated as an $\bN$-algebra by the $k+1$ elements $\psi_p^{\circ i}\circ d_p^{\circ j}$, and $\lttn{k}$ is a free module of rank $p^{\frac{k(k+1)}{2}}$ over it.
\end{lemma}
\begin{proof}
	A morphism $\varphi\:A_k\to\lttn{k}$ sending $x_{i,j}$ to 
	$\psi_p^{\circ i}\circ d_p^{\circ j}$ exists because we have
	\begin{align}
	\label{eq:local-485}	
	\begin{split}
	(\psi_p^{\circ i}\circ d_p^{\circ j})^p &= e^p\circ(\psi_p^{\circ i}\circ d_p^{\circ j}) \\
		&= (\psi_p+pd_p)\circ(\psi_p^{\circ i}\circ d_p^{\circ j}) \\
		&= \psi_p^{\circ (i+1)}\circ d_p^{\circ j} + p\psi_p^{\circ i}\circ d_p^{\circ (j+1)}.
		\end{split}
	\end{align}
	since $\psi_p$ commutes under composition with every element.
	Also, $\varphi$ is clearly surjective. 
	
	Let us now show injectivity. Since $A_k$ is a free $\bN$-module,
	it is enough to do so after tensoring with $\bQ$. Consider the diagram
	$$
	\xymatrix{
	\bQ[x_{0,0},\dots,x_{k,0}] \ar@{>->}^-{\text{incl}}[r] & \bQ\tn_\bN A_k \ar^-{\id\tn\varphi}[r] &
		\bQ\tn_\bN\Lambda_{\bN,(p),k}.
	}
	$$	
	Observe that the composition is injective because it sends the $x_{i,0}$ to the $\psi_p^{\circ i}$, which 
	are algebraically independent elements of $\bQ\tn_\bZ\Lambda_{\bZ,(p),k}$ 
	(and in fact are free $\bQ$-algebra generators). Therefore it
	is enough to show that the first inclusion is an equality. That is, 
	it is enough to show that the elements
	$x_{0,0},\dots,x_{k,0}$ generate $\bQ\tn_\bN A_k$ as a $\bQ$-algebra. 
	
	This is follows directly from the relations $x_{i,j}^p=x_{i+1,j}+px_{i,j+1}$. Indeed they imply that,
	for any $j$, if all $x_{i,j}$ lie in a given sub-$\bQ$-algebra, then so do all $x_{i,j+1}$.
	Therefore, the sub-$\bQ$-algebra generated by all $x_{i,0}$ contains all $x_{i,1}$, and hence all 
	$x_{i,2}$, and so on. Therefore it consists of all of $\bQ\tn_\bN A$.
\end{proof}

\begin{proposition}
	\label{pro:p-typical-lambda-k-is-free}
	\begin{enumerate}
		\item $\lttn{k}$ is free as an $\bN$-module.
		\item It is an $\bN$-model for $\Lambda_{\bZ,(p),k}$ as a co-$\bN$-algebra object in a 
			unique way.
		\item We have $\lttn{k}\circ\lttn{k'}\subseteq \lttn{k+k'}$.
	\end{enumerate}
\end{proposition}
\begin{proof}
	(1): This follows from~(\ref{lem:p-typical-model}).
	
	(2): The induced map $\bZ\tn_\bN\lttn{k}\to\Lambda_{\bZ,(p),k}$ is an injection, and it is a 
	surjection because $\Lambda_{\bZ,(p),k}$
	is generated as a $\bZ$-algebra by $e,d_p,\dots,d_p^{\circ k}$, all of which are contained
	in $\lttn{k}$. Therefore $\lttn{k}$ is a model for $\Lambda_{\bZ,(p),k}$ as an algebra.
	
	To show it is a model as a co-$\bN$-algebra object, it is enough to show 
		\begin{equation}
			\label{eq:local-123498}
		\Delta(x)\in \lttn{k}\tn_\bN\lttn{k}
		\end{equation}
	as $x$ runs over a set of $\bN$-algebra generators
	for $\lttn{k}$, and as $\Delta$ runs over the two coproducts $\Delta^+,\Delta^\times$.

	First consider the case $k=1$, where we have the generators $e,\psi_p,d_p$.
	Since $e$ and $\psi_p$ are both $\bN$-algebra-like elements, (\ref{eq:local-123498})
	holds for them. For $d_p$, it follows from the positivity of the coefficients in the 
	following equalities:
		\begin{align*}
			\Delta^+(d_p) &= d_p\tn 1 + 1\tn d_p + e\tn e \\
			\Delta^\times(d_p) &= d_p\tn\psi_p + \psi_p\tn d_p + pd_p\tn d_p.
		\end{align*}

 	Now consider the general case. Since $\lttn{1}$ is a co-$\bN$-algebra object in $\Alg{\bN}$,
	we can form $(\lttn{1})^{\bcp k}$. Consider the map
		$$
		(\lttn{1})^{\bcp k}\longmap (\Lambda_{\bZ,(p),1})^{\bcp k} \longmap \Lambda_{\bZ,(p),k}
		$$
	of co-$\bN$-algebra objects defined by 
	$f_1\bcp\cdots\bcp f_{k}\mapsto f_1\circ\cdots\circ f_{k}$.	Then its image
	is equal to $\lttn{k}$. Indeed, $\lttn{1}$ is generated by $e,\psi_p,d_p$, all of which commute
	with each other under composition; so the image is the sub-$\bN$-algebra generated by all 
	$\circ$-words in $e,\psi_p,d_p$ of length $k$, which is $\lttn{k}$. Thus we have a 
	surjection
		\begin{equation}
			\label{eq:local-598258}
		(\lttn{1})^{\bcp k}\longmap \lttn{k}
		\end{equation}
	of co-$\bN$-algebra objects, and hence a diagram
		$$
		\xymatrix{
		(\lttn{1})^{\bcp k}\ar[r]\ar^{\Delta}[d] 
			& \Lambda_{\bZ,(p),k}\ar^{\Delta}[d] \\
		(\lttn{1})^{\bcp k}\tn_\bN (\lttn{1})^{\bcp k} \ar[r] 
			& \Lambda_{\bZ,(p),k}\tn_\bN \Lambda_{\bZ,(p),k}
		}
		$$
	Therefore $\Delta(\lttn{k})$ is contained in the image of 
	$(\lttn{1})^{\bcp k}\tn_\bN(\lttn{1})^{\bcp k}$, which is contained in 
	$\lttn{k}\tn_\bN\lttn{k}$.
	Therefore $\lttn{k}$ is a model as a co-$\bN$-algebra object.
	
	(3): Because (\ref{eq:local-598258}) is surjective, $\lttn{k}\circ\lttn{k'}$ equals the image 
	of $$(\lttn{1})^{\bcp k}\bcp(\lttn{1})^{\bcp k'}\longmap\Lambda_{\bZ,(p)}$$ which is 
	$\lttn{k+k'}$.
\end{proof}

\subsection{$p$-typical Witt vectors and $\Lambda$-structures for semirings}
\label{subsec:p-typical-defs}
We can put everything above together in what should now be a familiar way.

Define 
	$$
	\Lambda_{\bN,(p)} := \colim_k \lttn{k}.
	$$
It is a flat model for $\Lambda_{\bZ,(p)}$ over $\bN$, as composition algebra. Write
	$$
	W_{(p)}(A) := \Alg{\bN}(\Lambda_{\bN,(p)},A)
	$$
for its Witt vector functor. The truncated version is
	$$
	W_{(p),k}(A) := \Alg{\bN}(\lttn{k},A),
	$$
and we have
	$$
	W_{(p)}(A) = \lim_k W_{(p),k}(A).
	$$
These Witt functors take values in $\bN$-algebras. If $A$ is a ring, then they
agree with the usual $p$-typical Witt vector rings.

Giving an action of $\Lambda_{\bN,(p)}$ on an $\bN$-algebra $A$ is equivalent to giving an 
$\bN$-algebra endomorphism $\psi_p\:A\to A$ and a set map $d_p\:A\to A$ satisfying the identities
	\begin{align}
		\label{eq:delta-semiring-axioms}
		\begin{split}
		\psi_p(d_p(x)) &= d_p(\psi_p(x)) \\
		\psi_p(x) + pd_p(x) &= x^p \\
		d_p(x+y) &= d_p(x) + d_p(y) + \sum_{i=1}^{p-1}\frac{1}{p}\binom{p}{i}x^i y^{p-i} \\
		d_p(xy) &= d_p(x)\psi_p(y) + \psi_p(x)d_p(y) + pd_p(x)d_p(y) \\
		d_p(0) &= 0 \\
		d_p(1) &= 0.
		\end{split}
	\end{align}

Observe that if $A$ is flat over $\bN$, then $\psi_p$ determines $d_p$, assuming it exists.
This is because additive cancellativity and $p$-cancellativity are flat-local properties,
by~(\ref{subsec:cancellative-modules}). Thus giving a $\Lambda_{\bN,(p)}$-structure on
a flat $\bN$-algebra is equivalent to  giving an $\bN$-algebra endomorphism $\psi_p$
lifting the Frobenius map on $\bZ/p\bZ\tn_\bN A$ and satisfying $\psi_p(x)\leq x^p$ for all $x$.
Here we write $a\leq b$ if there exists an element $c\in A$ such that $a+c=b$.

\subsection{Remark: a partition-like interpretation of the bases}

It is possible to give an interpretation of the bases of $\lttn{k}$ and $\Lambda_{\bN,(p)}$ in
the language of partitions. For simplicity of notation, let us write $\lttn{\infty}=\Lambda_{\bN,(p)}$
and hence allow $k=\infty$.

Over $\bZ$, there is a basis of $\Lambda_{\bZ,(p),k}$ given by monomials in
the $d_p^{\circ i}$ or the $\theta_{p^i}$, where $i\geq 0$. They can be indexed by usual partitions
$\lambda$ that are $p$-typical in the sense that all parts $\lambda_j$ are powers of $p$. For
example, one could use the family $\prod_j\theta_{\lambda_j}$. Equivalently, we could index them
by the $p$-typical multiplicity vectors $m\in\bigoplus_{i \geq 0} \bN$, where $m_i$ is the number
of $j$ such that $\lambda_j=p^i$.

On the other hand, the $\bN$-basis of $\lttn{k}$ consists of monomials 
	$$
	\prod_{i+j\leq k} (\psi_p^{\circ i}\circ d_p^{\circ j})^{m_{ij}},
	$$
such that when $i+j<k$, we have $m_{ij}<p$. Observe that such a monomial remains
a basis element in $\lttn{k+1}$ if and only if $m_{ij}< p$ when $i+j=k$.

We might think of the vector $m\in \bigoplus_{i,j\geq 0}\bN$ as the multiplicity vector of a
2-dimensional $p$-typical partition. Such a partition would be an expression of the form
$\sum_{i,j}m_{ij}p^{i+j}$. The basis for $\lttn{k}$ would then be indexed by all $2$-dimensional
$p$-typical partitions subject to the conditions that there are less than $p$
parts $(i,j)$ when $i+j<k$, and no parts $(i,j)$ when $i+j>k$.

\subsection{Relation to the multiple-prime theory}

Since $\psi_p,d_p\in\Lambda_\bN$, we have $\Lambda_{\bN,(p)}\subseteq\Lambda_\bN$.
This induces canonical algebra maps $W(A)\to W_{(p)}(A)$ for all $\bN$-algebras $A$.
In particular, for each $a\in A$, there is a Teichm\"uller lift $[a]\in W_{(p)}(A)$.
It is the image of the usual Teichm\"uller lift $[a]\in W(A)$.

On the other hand, we have $\Lambda_{\bN,(p)}\not\subseteq \Sch$, simply because
$\psi_p\not\in\Sch$. In particular, there is no functorial map $\Wsch(A)\to W_{(p)}(A)$ that agrees
with the usual one for rings.

\subsection{Some explicit descriptions of $W_{(p),k}(A)$}

The presentation~(\ref{map:lambda_p,k-presentation}) translates directly into finite 
descriptions of the Witt vectors of finite length:
	$$
	W_{(p),k}(A) = \bigsetof{(a_{i,j})\in 
		A^{\setof{(i,j)}{i+j\leq k}}}{a_{i,j}^p=a_{i+1,j}+pa_{i,j+1} \text{ for } i+j<k }.
	$$
For example, 
	$$
	W_{(p),1}(A) = \bigsetof{(a_{00},a_{10},a_{01})}{a_{00}^p=a_{10}+pa_{01}}
	$$
In general, $W_{(p),k}=\Spec(\lttn{k})$ is the locus in the $\bN$-scheme
$\bA^{\binom{k+2}{2}}_{\bN}$ defined by the $\binom{k+1}{2}$ relations in the algebra $A_k$ 
of~(\ref{subsec:p-typical-symmetric-functions}).

As usual with Witt vectors, the algebraic structure is not transparent when expressed in
coordinates. The simplest nontrivial example is $W_{(p),1}(A)$, where we have
	\begin{align*}
	(a_{00},a_{10},a_{01})+(b_{00},b_{10},b_{01}) 
		&= (a_{00}+b_{00},\, a_{10}+b_{10},\, a_{01}+b_{01}
			+\sum_{i=1}^{p-1}\frac{1}{p}\binom{p}{i}a_{00}^i b_{00}^{p-i}) \\
	(a_{00},a_{10},a_{01})(b_{00},b_{10},b_{01}) 
		&= (a_{00}b_{00},\, a_{10}b_{10},\, a_{10}b_{01} + a_{01} b_{10}+pa_{01}b_{01}) \\
	0 &= (0,0,0)\\
	1 &= (1,1,0).
	\end{align*}
This is just another expression of the formulas of (\ref{eq:delta-semiring-axioms}).

\subsection{$W_{(p),k}(A)$ when $A$ is contained in a ring}
In this case, we can ignore the relations in the presentation~(\ref{map:lambda_p,k-presentation})
and instead describe $W_{(p),k}(A)$ in terms of the usual $p$-typical Witt vector
ring $W_{(p),k}(\bZ\tn_\bN A)$ and
effectivity conditions corresponding to the generators $\psi_p^{\circ i}\circ d_p^{\circ j}$.
Indeed, a morphism $\lttn{k}\to A$ is equivalent to a morphism
$a\:\Lambda_{\bZ,(p),k}\to\bZ\tn_\bN A$ such that $a(\psi_p^{\circ i}\circ d_p^{\circ j})\in A$ for 
all $i,j$. Thus we have
	$$
	W_{(p),k}(A) = \bigsetof{a\in W_{(p),k}(\bZ\tn_\bN A)}
		{a(\psi_p^{\circ i}\circ d_p^{\circ j})\in A\text{ for } i+j\leq k}.
	$$
For instance, if $A$ is contained in a $\bQ$-algebra, then this permits a recursive description in
terms of ghost components: 
$W_{(p),k}(A)$ is the set of ghost vectors $\langle{a_0,\dots,a_k}\rangle\in A^{k+1}$ satisfying 
the following property:
	$$
	\langle{a_1,\dots,a_{k}}\rangle,\ 
	\Big\langle{\frac{a_0^p-a_1}{p},\dots,\frac{a_{k-1}^p-a_k}{p}}\Big\rangle \in W_{(p),k-1}(A).
	$$
Thus the conditions are that for all $i\geq 0$, the elements
	\begin{multline*}
	\frac{a_i^p-a_{i+1}}{p},\quad  \frac{1}{p}\Big(\frac{a_i^p-a_{i+1}}{p}\Big)^p 
		-\frac{1}{p}\Big(\frac{a_{i+1}^p-a_{i+2}}{p}\Big),\\ 
	\frac{1}{p}\Bigg(\frac{1}{p}\Big(\frac{a_i^p-a_{i+1}}{p}\Big)^p 
		-\frac{1}{p}\Big(\frac{a_{i+1}^p-a_{i+2}}{p}\Big)\Bigg)^p
	- \frac{1}{p}\Bigg(\frac{1}{p}\Big(\frac{a_{i+1}^p-a_{i+2}}{p}\Big)^p 
		-\frac{1}{p}\Big(\frac{a_{i+2}^p-a_{i+3}}{p}\Big)\Bigg)
	\end{multline*}
and so on lie in $A$.	

For example, $W_{(p),1}(\bRpl)$ is the set of ghost vectors $\langle{x,y}\rangle$ in $\bRplex{2}$ 
with $y \leq x^p$. Similarly one can show that $W_{(p),2}(\bRpl)$
is the set of ghost vectors in $\bRplex{3}$ of the form $\langle{a,a^p x,a^{p^2}y}\rangle$, where
	$$
	x\leq 1 \quad \text{ and } \quad 0 \leq x^p-y \leq \frac{(1-x)^p}{p^{p-1}}.
	$$
Thus the pre-image of $1$ under the projection $W_{(p),2}(\bRpl)\to\bRpl$ onto the first coordinate 
is the 2-simplex bounded by the curves $y=x^p$, $y=x^p-(1-x)^p/p^{p-1}$ and $y=0$  
in the $xy$-plane.

\subsection{Counterexample: The canonical map $W_{(p),k+1}(A)\to W_{(p),k}(A)$ is not generally 
surjective}

It is surjective when $A$ is a ring or when $k=0$. But when $k\geq 1$ and $A$ is general, it is not. It is
enough to check this in the universal case, when $A=\lttn{k}$. In other words, it is enough to show the
inclusion $\lttn{k}\to\lttn{k+1}$ has no retraction in the category $\Alg{\bN}$. So suppose $\varphi$ is such a
retraction. By~(\ref{eq:local-485}), we have 
	$$
	(\psi_p^{\circ i}\circ d_p^{\circ j})^p 
		= \psi_p^{\circ(i+1)}\circ d_p^{\circ j} + p\psi_p^{\circ i}\circ d_p^{\circ (j+1)}.
	$$
Now suppose $i+j=k$. Then since $\varphi$ is a retraction, we have
	$$
	(\psi_p^{\circ i}\circ d_p^{\circ j})^p = \varphi(\psi_p^{\circ i}\circ d_p^{\circ j})^p
		= \varphi(\psi_p^{\circ(i+1)}\circ d_p^{\circ j}) + p\varphi(\psi_p^{\circ i}\circ d_p^{\circ (j+1)}).
	$$
But by~(\ref{lem:p-typical-model}), the left-hand side is additively indecomposable in $\lttn{k}$, 
and so both terms on the right-hand side vanish. So we have
$\varphi(\psi_p^{\circ i}\circ d_p^{\circ (j+1)})=0$ whenever $i+j=k$. Taking $i=k-1$ and $j=1$ gives
	$$
	\varphi(\psi_p^{\circ (k-1)}\circ d_p)^p = 
		\varphi(\psi_p^{\circ k}\circ d_p) 
		+ p\varphi(\psi_p^{\circ (k-1)}\circ d_p^{\circ 2}) = 0.
	$$
But this is impossible because $\varphi$ is a retraction.

\subsection{Semirings and the infinite prime}
\label{subsec:positivity-philosophy}

In the theory of $\Lambda$-rings, a finite prime $p$ allows us to speak of two things: $p$-adic integrality,
which is a property, and Frobenius lifts at $p$, which are structures. The fundamental point of this chapter
is that it is reasonable for some purposes to think of positivity as $p$-adic integrality at the place
$p=\infty$. So the infinite prime plays the first role here but not the second. This is meager when compared
to the rich role the infinite prime plays elsewhere in number theory, such as the theory of automorphic
forms, but our approach does have the virtue that it allows us to treat the first role purely algebraically,
and hence scheme-theoretically, as we can for finite primes.

But it is natural to wonder whether there is some analogue of the second role for $p=\infty$ and
whether there is an $\infty$-typical theory that can be isolated from the rest of the primes.
One might hope that flows will appear here.

\section{On the possibility of other models}

So far, we have not been concerned with whether our $\bN$-models are the most
natural ones---their existence has been interesting enough. The purpose of this short section
is to raise some questions in this direction.

As discussed in~(\ref{subsec:representation-theory}), the connection between symmetric functions and arithmetic
algebraic geometry is explained by Wilkerson's theorem, which we interpret as saying that $\Lambda_\bZ$ is the
composition ring that controls commuting Frobenius lifts. It is natural to ask whether there are similar,
arithmetically satisfying descriptions of the composition algebras over $\bN$ we have considered. As explained
in (\ref{subsec:p-typical-defs}), there is a such a description in the $p$-typical case. It would be
interesting to find one for $\Lambda_\bN$ or $\Sch$. A less satisfying alternative would be to
single out $\Lambda_\bN$ and $\Sch$ among all $\bN$-models by some general properties, and at least this form
of the question admits a precise expression:

\begin{question}
Are $\Lambda_\bN$, $\Sch$, and possibly the $B_p$ (of question~\ref{ques:p-poly-func} 
in~\ref{subsec:representation-theory}) the only flat models for $\Lambda_\bZ$ over $\bN$?
If not, is $\Sch$ the minimal one? Is $\Lambda_\bN$ the maximal one?
Is $\Lambda_{\bN,(p)}$ the only flat model for $\Lambda_{\bZ,(p)}$ over $\bN$?
\end{question}

Whether we ask for models as composition algebras or models as co-$\bN$-algebra objects, I do not
know the answer. I do not even know the answer to analogous questions about integrality at the 
finite primes. For instance, is $\Lambda_{\bZ,(p)}$ the maximal  integral model for 
$\bZ[1/p]\tn_\bZ\Lambda_{\bZ,(p)}$?

\begin{question}
Over $\bQpl$, there is another model for $\Lambda_\bQ$, namely $\Psi_\bQpl$. Is there still another?
\end{question}

\begin{question}
Let $K$ be a number field, and let $T$ be a set of embeddings $K\to\bR$. Do the composition 
algebras $\Lambda_{\sO_K,E}$ of~(\ref{subsec:lambda-over-number-fields}) have models over the 
sub-$\bN$-algebra of $\sO_K$ consisting of elements that are nonnegative under all $\sigma\in T$?
\end{question}

We have seen that if $K=\bQ$ and $T$ consists of the unique embedding, the answer is yes in two
cases: when $E$ consists of all maximal ideals of $\bZ$ or when it consists of only one. I do not
know the answer in any other case, unless $T$ or $E$ is empty.

\section{$k$-Schur functions and truncated Witt vectors}
\label{sec:k-Schur}

Let $\Lambda_{\bZ,k}$ denote $\bZ[h_1,\dots,h_k]$. Thus, in the notation 
of~(\ref{subsec:truncation-sets}), we have $\Lambda_{\bZ,k}=\Lambda_{\bZ,S}$, where $S$ is
the truncation set $\{1,2,\dots,k\}$. The purpose of this section is to show how $k$-Schur
functions, a recent development in the theory of symmetric functions, allow us to give an
$\bN$-model $\Sch_k$ for $\Lambda_{\bZ,k}$ which approaches $\Sch$ as $k$ tends to infinity.
Unfortunately, $\Sch_k$ is only a model as a co-$\bN$-module object, and not as a co-$\bN$-algebra
object. This would seem to be fatal for any application of $k$-Schur functions to Witt vectors as
objects of arithmetic algebraic geometry. But they do have several properties that are good from
the point of view of Witt vectors, and there are several parallels with the $p$-typical
$\bN$-models of section~\ref{sec:p-typical-model}. The purpose of this humble section is just to enter the
details into the literature, in case they can be of use to anyone else.

\subsection{$k$-Schur functions and $\Sch_{k}$}
It is not possible to make a $\bZ$-basis for $\Lambda_{\bZ,k}$ out of usual Schur functions. This
is because there are only finitely many Schur functions in any given $\Lambda_{\bZ,k}$. But
Lapointe--Lascoux--Morse~\cite{Lapointe-Lascoux-Morse:tableau-atoms} discovered certain symmetric
functions that form a basis for $\Lambda_{\bZ,k}$ and are similar to Schur functions in many ways.
They call them \emph{$k$-Schur}\index{Schur function!Schur function@$k$-Schur function} functions and denote them $s^{(k)}_\lambda$, where $\lambda$ runs
over all partitions $(\lambda_1,\dots)$ which are \emph{$k$-bounded}\index{partition!$k$-bounded} in the sense that $\lambda_1\leq k$.

Our reference for $k$-Schur functions will be the book~\cite{Lam-et-al:k-Schur-book} by Lam
\emph{et al.}, in particular part 2, which was written by Morse, Schilling, and Zabrocki and 
is based on lectures by Lapointe and Morse. They
consider more than one definition of $k$-Schur function but conjecture that they
all agree. For definiteness, we will take
	$$
	s^{(k)}_\lambda:=s^{(k)}_\lambda[X;1]
	$$ 
as our definition, where $s^{(k)}_\lambda[X;t]$ is defined there in equation (3.16), p.\ 
81. See pp.\ 83--84 for a discussion of the relations with the other definitions.

Define
	$$
	\Sch_{k} := \bigoplus_\lambda \bN s^{(k)}_\lambda.
	$$
As explained in part 2, section 4.5 of~\cite{Lam-et-al:k-Schur-book}, the family 
$s^{(k)}_\lambda$ forms a
$\bZ$-basis for $\Sch_{k}$. In other words, $\Sch_k$ is free over $\bN$ and is a model for 
$\Lambda_{\bZ,k}$ over $\bN$ as a module.

\begin{proposition}
\label{pro:k-Schur-facts}
For $k\geq 0$, we have
	\begin{enumerate}
	\item $\Sch_{k}$ is a model for $\Lambda_{\bZ,k}$ as a co-$\bN$-module object 
		in $\Alg{\bN}$,
	\item $s^{(k)}_\lambda=s_\lambda$, if $k\geq \lambda_1+l-1$, where $l$ is the length of 
		$\lambda$,
	\item $\Sch_{k}\subseteq \Sch_{k+1}$ and $\Sch_{k}\subseteq \Sch$,
	\item $\Sch_k$ is finitely presented as an $\bN$-algebra.
	\end{enumerate}
\end{proposition}

\begin{proof}
These are mostly just restatements of results collected in part 2, chapter 4 of the 
book~\cite{Lam-et-al:k-Schur-book}.

(1): This follows from corollaries 8.1 and 8.2 of Lam~\cite{Lam:Schubert-polys-and-affine-Grass}.
See sections 4.7 and 4.8 of part 2 of~\cite{Lam-et-al:k-Schur-book}.

(2): This is property 39 of Lapointe--Morse~\cite{Lapointe-Morse:k-tableau}.
The result under the stronger assumption $k\geq\sum_i\lambda_i$ is discussed in part 2, section 4.1 
of~\cite{Lam-et-al:k-Schur-book}. 

(3): The first statement is theorem 2 of Lam--Lapointe--Morse--Shimozono~\cite{LLMS:k-shape}.
See section 4.10 of part 2 of~\cite{Lam-et-al:k-Schur-book}.
The second statement follows from the first, together with part (2).

(4): This follows from the multiplication rule established in theorem 40 of
Lapointe--Morse~\cite{Lapointe-Morse:k-tableau}. See section 4.6 of part 2
of~\cite{Lam-et-al:k-Schur-book}. As a module over the $\bN$-algebra generated by the $k$
$k$-rectangular $k$-Schur functions, $\Sch_k$ is freely generated by the $k!$ $k$-irreducible
$k$-Schur functions. Further, the $k$-rectangular $k$-Schur functions are algebraically
independent. So $\Sch_k$ can be generated by $k!+k$ elements with $\binom{k!+1}{2}$ quadratic
relations.
\end{proof}

\subsection{Remark} 
There are some similarities between $\Sch_k$ and $\lttn{k}$. Compare the preceding
proposition with section~\ref{sec:p-typical-model}, and especially the presentation of $\Sch_k$
mentioned in the proof of (4) above with the presentation of~(\ref{map:lambda_p,k-presentation}).

\subsection{Truncated Schur--Witt vectors for semirings}
For any $\bN$-algebra $A$, define
	$$
	\Wsch_k(A):=\Alg{\bN}(\Sch_k,A).
	$$
It follows from~(\ref{pro:k-Schur-facts})(4) that $\Wsch_k(A)$ can be described as the subset of a 
finite-dimensional affine space $A^N$ satisfying a finite list of equations. 

By~(\ref{pro:k-Schur-facts})(2)--(3), we have $\Sch = \colim_k\Sch_k$ 
and $\Wsch(A) = \lim_k \Wsch_k(A)$.

It follows from~(\ref{pro:k-Schur-facts})(1) that $\Wsch_k(A)$ inherits an $\bN$-module structure
and that $\Wsch_k(A)=W_k(A)$ when $A$ is a ring.
Unlike in the case when $A$ is a ring, $\Wsch_k(A)$ does not generally inherit an $\bN$-algebra structure. 

\subsection{Counterexample: $\Sch_k$ is not a co-$\bN$-algebra object}
\label{subsec:Sch_k-not-a-co-algebra}

It is for $k\leq 2$, but as Luc Lapointe informed me, we have
$\Delta^\times(s^{(3)}_{22})\not\in\Sch_3\tn\Sch_3$, and so it fails for $k=3$.
This can be checked by hand using the following equalities:
\begin{align*}
	12s^{(3)}_{22} &= \psi_1^4 + 3\psi_2^2 -4\psi_1\psi_3 \\
	\psi_1^4 &= s^{(3)}_{1111} + 2s^{(3)}_{211} + 2s^{(3)}_{22} +s^{(3)}_{31} \\
	\psi_2^2 &= s^{(3)}_{1111} - 2s^{(3)}_{211} + 2s^{(3)}_{22} +s^{(3)}_{31} \\
	\psi_1\psi_3 &= s^{(3)}_{1111} - s^{(3)}_{211} - s^{(3)}_{22} +s^{(3)}_{31}.
\end{align*}
On the other hand, $\Delta^{\times}(s^{(3)}_{22})$ is contained in $\Sch_4\tn\Sch_4$. This is just
because it is an $\bN$-linear combination of elements of the form $s_\lambda\tn s_\mu$, where
$\lambda$ and $\mu$ are partitions of $2+2=4$; so we have $s_\lambda=s^{(4)}_\lambda$ and
$s_\mu=s^{(4)}_\mu$ for all $\lambda,\mu$ in question.

It is also not true that $\Sch_k\circ\Sch_{l}\subseteq\Sch_{kl}$ for all $k,l$. According to my
computations, it is true if $k,l\leq 3$ and $(k,l)\neq (3,3)$. But for $k=l=3$, it fails: the
coefficient of $s^{(9)}_{6331111}$ in $s^{(3)}_{22}\circ s^{(3)}_{22}$ is $-1$. In fact, we have
$\Sch_3\circ\Sch_{3}\not\subseteq\Sch_{11}$ and $\Sch_3\circ\Sch_{3}\subseteq\Sch_{12}$.

\section{Remarks on absolute algebraic geometry}

This volume is a collection of contributions on the theme of the mythical field with one element.
One can see this chapter from that point of view, although I have so far avoided making the
connection. There are two natural approaches to rigidifying the category of rings---one can look
for models over $\bN$, or one can add structure, such as a $\Lambda_\bZ$-ring structure, which we
think of as descent data to the absolute point~\cite{Borger:LRFOE}. In this chapter, we have
combined the two. I do not have much more to say about the philosophy of the field with one element
than I already have said in~\cite{Borger:LRFOE}, but this way of thinking does suggest some
mathematical questions.

\begin{question}
Is it possible to extend the constructions $W$ and $\Wsch$ to non-affine $\bN$-schemes?
What about their adjoints $A\mapsto \Lambda_\bN\bcp A$ and $A\mapsto \Sch\bcp A$?
\end{question}

Over $\bZ$, this was done in my paper~\cite{Borger:BGWV-II}, but there are several complications over $\bN$.
The most important is that over $\bZ$ I used Witt vectors of finite length, because it is better to think of
$W(A)$ as a projective system of discrete rings, rather than actually taking the limit. But there is not yet
any finite-length version of the big Witt vector functor for $\bN$-algebras. On the other hand, we do have
finite-length $p$-typical functors for $\bN$-algebras available; so it is probably easier to make immediate
progress there.

A similar question is whether the concept of a $\Lambda_\bN$-structure or
$\Sch$-structure can be extended to nonaffine $\bN$-schemes. Over $\bZ$, this is done using
the functor $W^*=\colim_n W_n^*$, where $W_n^*$ is the extension of $W_n$ to nonaffine schemes.
(One could also use its right adjoint $W_*=\lim_n W_{n*}$.) So the
two questions are indeed closely related. The following question is a natural guide:

\begin{question}
Let $X$ be a $\Lambda_\bN$-scheme that is flat and locally of finite presentation 
over $\bN$. Does there exist a toric variety $Y$ over $\bN$ and a surjective $\Lambda_\bN$-morphism
$Y\to X$?
\end{question}

This requires some explanation. By a {\em toric variety over $\bN$}\index{toric!variety over $\bN$}, I mean an $\bN$-scheme that can be
formed by gluing together affine $\bN$-schemes of the form $\Spec(\bN[M])$, where $M$ is a
commutative monoid, along other schemes of the same form, where all the gluing maps are induced by
maps of the monoids. Surjectivity of a morphism of $\bN$-schemes can be understood in the sense of
the Zariski topos. Finally, while a $\Lambda_\bN$-structure has at the moment no precise meaning
for nonaffine $X$, it is possible to strengthen the question to a precise one that is still open.
For instance, we can require only that the base change of $X$ to $\bZ$ admit a
$\Lambda_\bZ$-structure.

\newpage
\frenchspacing
\bibliography{references}
\bibliographystyle{plain}

\printindex

\newpage\mbox{}\newpage
\end{document}